\documentclass[12pt]{amsart}

\usepackage[utf8]{inputenc}
\usepackage[T1]{fontenc}
\usepackage{upref}
\usepackage{amssymb}
\usepackage{lmodern}
\usepackage{enumitem}
\usepackage{tikz-cd}

\theoremstyle{plain}
\newtheorem{lemma}{Lemma}[section]
\newtheorem{proposition}[lemma]{Proposition}
\newtheorem{theorem}[lemma]{Theorem}
\newtheorem{corollary}[lemma]{Corollary}
\theoremstyle{definition}
\newtheorem{definition}[lemma]{Definition}
\newtheorem{remark}[lemma]{Remark}
\newtheorem{example}[lemma]{Example}

\DeclareMathOperator{\Exp}{Exp}
\DeclareMathOperator{\rad}{rad}
\DeclareMathOperator{\basissupp}{supp}
\DeclareMathOperator{\suppvar}{supp}
\DeclareMathOperator{\polysupp}{supp}
\DeclareMathOperator{\id}{id}
\DeclareMathOperator{\linrank}{rank}
\DeclareMathOperator{\basisrank}{rank}
\DeclareMathOperator{\initial}{in}
\DeclareMathOperator{\ZL}{Z}
\DeclareMathOperator{\I}{I}
\DeclareMathOperator{\normalset}{N}
\DeclareMathOperator{\Hankel}{H}
\DeclareMathOperator{\Toeplitz}{T}
\DeclareMathOperator{\Vandermonde}{V}
\DeclareMathOperator{\Cat}{C}
\DeclareMathOperator{\Hom}{Hom}
\DeclareMathOperator{\End}{End}
\DeclareMathOperator{\pspec}{\sigma_p}
\DeclareMathOperator{\cont}{C}
\DeclareMathOperator{\bigTheta}{\Theta}
\DeclareMathOperator{\laplace}{\Delta}
\DeclareMathOperator{\Pronyindex}{ind}
\DeclareMathOperator{\Mon}{Mon}
\DeclareMathOperator{\border}{\partial}
\DeclareMathOperator{\gaussian}{g}

\DeclareMathOperator{\chebexp}{txp}
\DeclareMathOperator{\chebExp}{Txp}
\providecommand{\cheb}{\textnormal{T}}

\providecommand{\ev}[2]{\mathop{\textnormal{ev}_{#1}^{#2}}}
\providecommand{\abs}[1]{\mathopen{\lvert}#1\mathclose{\rvert}}
\providecommand{\norm}[1]{\mathopen{\lVert}#1\mathclose{\rVert}}

\providecommand{\euler}{\textnormal{e}}
\providecommand{\unit}{\textnormal{e}}
\providecommand{\qquot}[1]{``#1''}
\providecommand{\nat}{\mathbb{N}}
\providecommand{\integ}{\mathbb{Z}}
\providecommand{\rat}{\mathbb{Q}}
\providecommand{\real}{\mathbb{R}}
\providecommand{\complex}{\mathbb{C}}
\providecommand{\torus}{\mathbb{T}}
\providecommand{\sphere}{\mathbb{S}}
\providecommand{\lcrc}[2]{\mathopen{[}#1,#2\mathclose{]}}
\providecommand{\prp}[2]{#1\left.\left.\middle\vert\right.\right.#2}
\providecommand{\sumprp}[2]{#1\left.\left.\vphantom{\sum}\middle\vert\right.\right.#2}
\providecommand{\set}[1]{\mathopen{\lbrace}{#1}\mathclose{\rbrace}}
\providecommand{\sumset}[1]{\mathopen{\Big\lbrace}{#1}\mathclose{\Big\rbrace}}
\providecommand{\pset}[2]{\set{\prp{#1}{#2}}}
\providecommand{\psumset}[2]{\sumset{\sumprp{#1}{#2}}}
\providecommand{\ringad}[1]{\mathopen{\lbrack}#1\mathclose{\rbrack}}
\providecommand{\fieldad}[1]{\mathopen{(}#1\mathclose{)}}
\providecommand{\restr}[2]{#1\vert_{#2}}
\providecommand{\of}[1]{\mathopen{(}#1\mathclose{)}}
\providecommand{\sumof}[1]{\mathopen{\Big(}#1\mathclose{\Big)}}
\providecommand{\sumoftext}[1]{\mathopen{\big(}#1\mathclose{\big)}}
\providecommand{\rb}[1]{\mathopen{(}#1\mathclose{)}}
\providecommand{\sumrb}[1]{\mathopen{\Big(}#1\mathclose{\Big)}}
\providecommand{\divides}{\mathrel{{\mid}}}
\providecommand{\mgen}[1]{\mathopen{\langle}{#1}\mathclose{\rangle}}
\providecommand{\pmgen}[2]{\mgen{\prp{#1}{#2}}}
\providecommand{\igen}[1]{\mathopen{\langle}{#1}\mathclose{\rangle}}
\providecommand{\mul}{\mathbin{{\cdot}}}
\providecommand{\comp}{\mathbin{{\circ}}}
\providecommand{\transpose}[1]{{#1}^{\top}}
\providecommand{\mimpl}{\mathrel{{\Rightarrow}}}
\providecommand{\apriori}{a~priori}
\providecommand{\Apriori}{A~priori}
\providecommand{\confer}{cf.}
\providecommand{\wlogen}{w.l.o.g.}
\providecommand{\wrt}{w.r.t.}
\providecommand{\idest}{i.e.}
\providecommand{\eg}{e.g.}
\providecommand{\bigdirsum}{\bigoplus}
\providecommand{\defeq}{\mathrel{{\mathop:}{=}}}
\providecommand{\emb}{\mathrel{{\hookrightarrow}}}
\providecommand{\trb}[1]{\textup{(}#1\textup{)}}
\providecommand{\tbrck}[1]{\textup{[}#1\textup{]}}
\providecommand{\card}[1]{\mathopen{\lvert}{#1}\mathclose{\rvert}}
\providecommand{\isom}{\mathrel{{\cong}}}
\providecommand{\lbar}[1]{\overline{#1}}
\providecommand{\W}{\textnormal{w}}
\providecommand{\X}{\textnormal{x}}
\providecommand{\Y}{\textnormal{y}}
\providecommand{\Z}{\textnormal{z}}
\providecommand{\catP}{\calP}
\providecommand{\calC}{\mathcal{C}}
\providecommand{\calH}{\mathcal{H}}
\providecommand{\calI}{\mathcal{I}}
\providecommand{\calJ}{\mathcal{J}}
\providecommand{\calM}{\mathcal{M}}
\providecommand{\calO}{\mathcal{O}}
\providecommand{\calP}{\mathcal{P}}
\providecommand{\calT}{\mathcal{T}}
\providecommand{\degrevlex}{\textnormal{degrevlex}}
\providecommand{\FC}{\calC}
\providecommand{\FH}{\calH}
\providecommand{\FI}{\calI}
\providecommand{\FJ}{\calJ}
\providecommand{\FM}{\calM}
\providecommand{\FT}{\calT}
\providecommand{\harmonhomog}{\textnormal{harmH}}
\providecommand{\sphereharmon}{\textnormal{SH}}
\providecommand{\claimmimplnocolon}[2]{$\text{#1}\mimpl\text{#2}$}
\providecommand{\claimmimpl}[2]{\claimmimplnocolon{#1}{#2}:}

\newenvironment*{tlist}{\begin{enumerate}[label=\textup{(\alph*)},ref=\textup{(\alph*)}]}{\end{enumerate}}
\newenvironment*{tlist2}{\begin{enumerate}[label=\textup{(\arabic*)},ref=\textup{(\arabic*)}]}{\end{enumerate}}
\newenvironment*{ifflist}{\begin{enumerate}[label=\textup{(\roman*)},ref=\textup{(\roman*)}]}{\end{enumerate}}
\newenvironment*{qlist}{\begin{enumerate}[label=\textup{($\text{Q}_{\arabic*}$)},ref=\textup{($\text{Q}_{\arabic*}$)}]}{\end{enumerate}}

\title{Learning algebraic decompositions using Prony structures}

\author{Stefan~Kunis}
\address{Institute of Mathematics
and
Research Center of Cellular Nanoanalytics,
Osnabr\"uck University,
49069~Osnabr\"uck,
Germany}
\email{skunis@uos.de}

\author{Tim~R\"omer}
\address{Institute of Mathematics,
Osnabr\"uck University,
49069~Osnabr\"uck,
Germany}
\email{troemer@uos.de}

\author{Ulrich~von~der~Ohe}
\address{Dipartimento di Matematica,
Universit\`a degli Studi di Genova,
Via Dodecaneso~35,
16146~Genova,
Italy;
Marie Sklodowska-Curie fellow
of the
Istituto Nazionale di Alta Matematica}
\email{vonderohe@dima.unige.it}
\thanks{%
The third author was supported
by an
INdAM-DP-COFUND-2015/Marie~Sk\l{}odowska-Curie Actions
scholarship,
grant number~713485.
We
gratefully acknowledge
support by
the
MIUR-DAAD Joint Mobility Program
\trb{\qquot{PPP~Italien}}.%
}


\date{\today}

\begin{document}

\begin{abstract}
We propose an algebraic framework generalizing several variants of Prony's method and explaining their relations.
This includes Hankel and Toeplitz variants of Prony's method for the decomposition of
multivariate exponential sums,
polynomials
\trb{\wrt~the monomial and Chebyshev bases},
Gau\ss{}ian sums,
spherical harmonic sums,
taking also into account whether they have their support on an algebraic set.
\end{abstract}

\maketitle

\section*{Introduction}

Learning decompositions of functions from their evaluations in terms of a given basis
and similar questions like the moment problem
are fundamental tasks in signal processing and related areas.

In~1795,
Prony proposed an algebraic approach
to give an answer to such a question
in the case of univariate exponential sums~%
\cite{Ric95}.
Classic applications of Prony's method include for example
Sylvester's method for Waring decompositions of binary forms~%
\cite{Syl51a,Syl51b}
and
Pad\'e approximation~%
\cite{WM63}.
Since then these tools have been further developed~%
\cite{PT10,PP13,PT13b,SP20+},
new applications have been found
\trb{see,~\eg,~%
\cite[Section~\textup{2.2}]{KL03}
for connections to the
Berlekamp-Massey algorithm},
and recently also
advances have been made on multivariate versions.
Direct attempts can be found in,~\eg,~%
\cite{PT13a,PPS15,KPRv16,Sau17,Mou18},
for methods based on projections to univariate exponential sums see,~\eg,~%
\cite{DI15,DI17,CL18}.
A numerical variant can be found in,~\eg,~%
\cite{EKPR19},
and further related results and applications in,~\eg,~%
\cite{DG91,GKS91,LS95,Cuy99,KY07,PT14,CH15+,CH18,JLM19,HS20+}.

As important as \trb{approximate} algorithms undeniably are in practice,
at its core Prony's method is of a purely algebraic nature
which is the point of view of this article.
We introduce a general algebraic framework
called Prony structures
for reconstruction methods modeled after Prony's original idea.
Our approach allows a simultaneous treatment
of decomposition problems
in particular for
multivariate exponential sums,
polynomials
\trb{\wrt~the monomial and Chebyshev bases},
Gau\ss{}ian sums,
and
eigenvector sums of linear operators.

To describe the main task,
consider a vector space~$V$~of functions with a distinguished basis~$B$.
The goal is to decompose an arbitrary function $f\in V$ into a linear combination of basis elements.
As a constraint for this it is only allowed to use evaluations of~$f$.

In typical Prony situations one has a way to identify basis elements with points in an affine space.
For example,
in the case of exponential sums the basis function~$\exp_b$~is identified with its base point~$b\in\complex^n$.
It is this identification that allows to describe the support of~$f$,~\idest~%
the used basis elements in the decomposition,
by polynomial equations.
A key idea of Prony is to construct
Hankel
\trb{or Toeplitz}
matrices using evaluations of~$f$
to obtain the desired data from their kernels.

In our framework we assume that an identification as above is given
as part of the initial data.
Then suitable sequences of matrices are computed from evaluations of~$f$
which are constructed in a way such that their kernels eventually have to yield
systems of polynomial equations to determine the support of~$f$.

The article is organized as follows.
In
Section~\ref{section:prony-structures}
we fix the setup,
some notation,
and introduce our main definition of a Prony structure.
Besides the function space and the basis
as key parts of the data
it consists of families of matrices
and associated ideals
defined by their kernels.
These ideals are then used to attack
the decomposition problem.
We also recall briefly,
as a special case,
the fundamental example
of Prony's classic method.

In
Section~\ref{section:evaluation-map}
we
discuss
properties of
evaluation maps on vector spaces of polynomials
and their kernels,
see for
example~%
\cite{KPRv16}.
As one of our main results,
we prove in
Theorem~\ref{theorem:prony-structure-characterization}
a very useful characterization of Prony structures
in terms of factorizations through evaluation maps.

It can be seen that given some mild assumptions
the ideals of a Prony structure are zero-dimensional and radical
\trb{see
Corollary~\ref{corollary:asymptotically-radical}},
which leads to the natural question
to provide sufficient conditions
which guarantee that the ideals of kernels of evaluation maps have this property.

In
Section~\ref{section:evaluation-map-and-moeller}
we study this problem.
The main result of this section
\trb{Theorem~\ref{theorem:moeller-ideal-basis}}
proves a theorem of M\"oller
on Gr\"obner bases of zero-dimensional radical ideals
with interesting consequences for Prony structures.

In
Section~\ref{section:exponential-sums}
we discuss in particular in
Theorem~\ref{theorem:exponential-sums}
fundamental examples of Prony structures
based on the Hankel and Toeplitz matrices
defined by
exponential sums,
see for example~%
\cite{Mou18}
for their use in classic situations
related to Prony's methods.

Known reconstruction techniques
can be used for
sums of eigenvector of linear operators~%
\cite{PP13},
polynomials
\trb{\wrt~the monomial and various types of Chebyshev bases}~%
\cite{BT88,LS95,KY07,PT14,Mou18},
and
multivariate Gau\ss{}ians~%
\cite{PPS15}.
In
Section~\ref{section:applications}
we will see
in particular
that they
arise from
Prony structures
related to those for exponential sums.
In this section we also show relations between
the framework of Prony structures and previously known frameworks
for character~%
\cite{DG91}
and
eigenfunction sums~%
\cite{GKS91,PP13}.

\Apriori~knowledge can be that functions are supported
for example on a torus or a sphere,
see,~\eg,~%
\cite{KMPv18,KMv19}.
Classic techniques do not take this additional information into account.
As a novel approach we extend the notion of Prony structures
for functions supported on algebraic sets
to a relative version
in
Section~\ref{section:relative-prony-structures}.
A first key result is
a characterization of such structures in
Theorem~\ref{theorem:relative-prony-structure-characterization}.
In
Theorem~\ref{theorem:relative-prony-structures-construction}
and its corollaries
we discuss how to obtain Prony structures in this relative case.
Main examples include relative Prony structures for spaces of spherical harmonics.

Already in the existing literature,
projection techniques are used to
apply Prony's method,~see,~\eg~%
\cite{DI15,DI17,CL18}.
Related to this idea
is an observation in
Section~\ref{section:applications}
that a Prony structure may be
\qquot{induced} by another one on a different vector space.
The systematic point of view of these phenomena
is given by maps between Prony structures,
which we introduce in
Section~\ref{section:prony-maps}.
We discuss
projection methods,
Gau\ss{}ian sums
and other examples in terms of such maps.

\vfill

\textsc{Acknowledgments.}
We are grateful towards
H.\,M.~M\"oller
for inspiring discussions
related to these results,
in particular for allowing us to
include
Theorem~\ref{theorem:moeller-ideal-basis}
and its proof~%
\cite{Moe17}.
The third author is grateful for the warm hospitality
he received when visiting Osnabr\"uck on several occasions.
We thank the referees for their valuable remarks
and additional pointers to the literature
which led to considerable improvements of the article.

\section{Prony structures}%
\label{section:prony-structures}

Motivated by Prony's reconstruction method as well as its recent generalizations
we introduce a framework that enables us to treat several
of these variants simultaneously
and which can be applied in various contexts.
In this section
we begin by fixing some notation regarding evaluation maps for polynomials,
and then make our main definition of Prony structures.
The key point is to give a general formal setting
that captures the essence of Prony's method with the aim of
laying the foundation for a structural theory.

\begin{definition}%
\label{definition:evaluation-map}
Let~$K$~be a field,
$n\in\nat$,
$S\defeq K\ringad{\X}\defeq K\ringad{\X_1,\dots,\X_n}$,
and for
an arbitrary subset $D\subseteq\nat^n$
let
$S_D\defeq\mgen{\X^D}_K=\pmgen{\X^\alpha}{\alpha\in D}_K$.
For $X\subseteq K^n$
define
\begin{equation*}
\ev{D}{X}\colon S_D\to K^X\text{,}
\quad
p\mapsto\rb{p\of{x}}_{x\in X}\text{,}
\end{equation*}
and
\begin{equation*}
\I_D\of{X}\defeq\ker\of{\ev{D}{X}}\text{.}
\end{equation*}
We call~$\ev{D}{X}$~the
\emph{evaluation map at~$X$}
and~$\I_D\of{X}$~the
\emph{vanishing space of~$X$}~%
\wrt~$S_D$.
\end{definition}

Observe that
for $D=\nat^n$ we just have $S_D=S$
and
$\I\of{X}=\I_{\nat^n}\of{X}$
is the usual
\emph{vanishing ideal of~$X$}.
In this special situation we also set
${\ev{}{X}}\defeq{\ev{\nat^n}{X}}$.
Note that in general
we have
\begin{equation*}
\I_D\of{X}=\I\of{X}\cap S_D\text{.}
\end{equation*}
In
Section~\ref{section:evaluation-map}
we will state all results on evaluation maps and their kernels
that are relevant for this note.

In order to characterize basis elements of a vector space~$V$
through systems of polynomial equations
we need a way to identify them with points.
This will be achieved by an injection~$u$
as in the following definition.

\begin{definition}%
\label{definition:u-support}
Let~$F$~be a field,
$V$~be an $F$-vector space,
and~$B$~be an $F$-basis of~$V$.
For $f\in V$,
$f=\sum_{i=1}^rf_ib_i$ with
$f_1,\dots,f_r\in F\setminus\set{0}$
and
distinct $b_1,\dots,b_r\in B$,
let
\begin{equation*}
\text{$\basissupp_B\of{f}\defeq\set{b_1,\dots,b_r}$
\quad
and
\quad
$\basisrank_B\of{f}\defeq\card{\basissupp_B\of{f}}=r$}
\end{equation*}
denote the
\emph{support of~$f$}
and
\emph{rank of~$f$}~%
\trb{\wrt~$B$},
respectively.
For a field~$K$, $n\in\nat$,
and an injective map
$u\colon B\to K^n$ let
\begin{equation*}
\suppvar_u\of{f}\defeq \set{u\of{b_1},\dots,u\of{b_r}}\text{.}
\end{equation*}
We call $\suppvar_u\of{f}\subseteq K^n$
the
\emph{$u$-support}
and its elements the
\emph{support labels of~$f$}.
\end{definition}

In many situations we will choose $K=F$,
but for reasons of flexibility we
allow the choice of possibly different fields.
Unless mentioned otherwise,
we will assume
that~$F$,~$V$, $B$,~$K$,~$n$, and~$u$~are given as in
Definition~\ref{definition:u-support}.
In the following definition we introduce
the central notion of a Prony structure.

\begin{definition}%
\label{definition:prony-structure}
Given the setup of
Definition~\ref{definition:u-support},
let
$\FI=\rb{\FI_d}_{d\in\nat}$
be a sequence of finite sets
and
$\FJ=\rb{\FJ_d}_{d\in\nat}$
be a sequence of finite subsets of~$\nat^n$.
Let
$f\in V$
and
\begin{equation*}
P\of{f}=\rb{P_d\of{f}}_{d\in\nat}\in\prod_{d\in\nat}K^{\FI_d\times\FJ_d}\text{,}
\end{equation*}
\idest,~a family of matrices with
$P_d\of{f}\in K^{\FI_d\times\FJ_d}$
for all $d\in\nat$.
We call~$P\of{f}$ a
\emph{Prony structure for~$f$}
if there is a $c\in\nat$
such that for all $d\in\nat$
with $d\ge c$
one has
\begin{equation}%
\label{equation:definition:prony-structure}
\text{$\ZL\of{\ker P_d\of{f}}=\suppvar_u\of{f}$
\quad
and
\quad
$\I_{\FJ_d}\of{\suppvar_u\of{f}}\subseteq\ker\of{P_d\of{f}}$.}
\end{equation}
Here we identify
$p\in\ker P_d\of{f}\subseteq K^{\FJ_d}$
with the polynomial
$\sum_{\alpha\in\FJ_d}p_\alpha\X^\alpha\in K\ringad{\X_1,\dots,\X_n}$
and
$\ZL\of{\cdot}$ takes the zero~locus of a set of polynomials.
See
Remark~\ref{remark:quasi-prony-structures}
for a discussion of the second condition,
which is not implied by the other.

The least $c\in\nat$
such that
the conditions in~\eqref{equation:definition:prony-structure}
hold for all $d\ge c$
is called
\emph{Prony index of~$f$}
or simply
\emph{$P$-index of~$f$},
denoted by
$\Pronyindex_P\of{f}$.

If for every $f\in V$ a Prony structure~$P\of{f}$~for~$f$~is given,
then we call~$P$~a
\emph{Prony structure on~$V$}.
\end{definition}

\begin{remark}
A key point of a Prony structure~$P$ on~$V$
is that the idea of Prony's method works,
\idest~to compute
the support of a given $f\in V$~\wrt~the basis~$B$
through a system of polynomial equations.
More precisely,
one can
perform the following \trb{pseudo-}algorithm:
\begin{enumerate}
\item
Choose
$d\in\nat$.
\item
Determine
$P_d\of{f}\in K^{\FI_d\times\FJ_d}$.
\item
Compute
$U\defeq\ker P_d\of{f}\subseteq K^{\FJ_d}$.
\item
Embed
$U\subseteq K\ringad{\X_1,\dots,\X_n}$.
\item
Compute
$Z\defeq\ZL\of{U}\subseteq K^n$.
\item
Compute
$u^{-1}\of{Z}\subseteq B$.
\end{enumerate}
If~$d$~is chosen large enough,
then the zero~locus~$Z$~is the $u$-support
and~$u^{-1}\of{Z}$~is the support of~$f$
\trb{and in particular these sets are finite}.
Note that for this strategy to work
it is important that the matrices~$P_d\of{f}$
can be determined from \qquot{standard information} on~$f$
\trb{such as evaluations if~$f$~is a function},
in particular without already knowing the support;
see also
Remark~\ref{remark:trivial-prony-structure}.
Often computation of the zero~locus
as well as a good choice of~$d$
turn out to be problematic steps.

In classic situations of Prony's method
the non-zero coefficients of~$f$~\wrt~$B$
can be computed in an additional step
by solving a system of linear equations
involving only standard information;
this system is finite since one has already computed the support.
We omit the discussion of this step here and in the following.
\end{remark}

Common options for the sequence~$\FJ=\rb{\FJ_d}_{d\in\nat}$
are
$\FJ=\FT$,
$\FJ=\FM$,
or
$\FJ=\FC$,
where
\begin{align*}
\FT_d
&\defeq
\psumset{\alpha\in\nat^n}{\sum_{j=1}^n\alpha_j\le d}\text{,}\\
\FM_d
&\defeq
\pset{\alpha\in\nat^n}{\max\pset{\alpha_j}{j=1,\dots,n}\le d}\text{,}\\
\text{and}
\quad
\FC_d
&\defeq
\psumset{\alpha\in\nat^n}{\prod_{j=1}^n\rb{\alpha_j+1}\le d}\text{.}
\end{align*}
Under the identification of~$K^{\FJ_d}$ with polynomials,
in $S=K\ringad{\X_1,\dots,\X_n}$
the choice
$\FJ=\FT$~corresponds to the subvector space of polynomials of
\emph{total degree}
at most~$d$,
and
$\FJ=\FM$~corresponds to the subvector space of polynomials of
\emph{maximal degree}
at most~$d$.
Choosing
$\FJ=\FC_d$,
the non-negative orthant of the
\emph{hyperbolic cross}
of order~$d$,
gives rise to a space of polynomials that is particularly
well-suited for zero-testing and interpolation of polynomials.
The earliest use of~$\FC_d$~in the context of Prony-like methods
that we are aware of is in articles by
Clausen, Dress, Grabmeier, and Karpinski~%
\cite{CDGK91}
and by
Dress and Grabmeier~%
\cite{DG91}.
For more recent applications see in particular
Sauer~%
\cite{Sau18}
and the preprint
Hubert-Singer~%
\cite{HS20+}.

Often one chooses $\FI=\FJ$,
$\FI_d=\FJ_{d-1}$,
or a similar relation between~$\FI$~and~$\FJ$.

We will also use the notation
\begin{equation*}
\text{$S_{\le d}\defeq S_{\FT_d}=\pmgen{\X^\alpha}{\alpha\in\FT_d}_K$,
\quad
${\ev{\le d}{X}}\defeq{\ev{\FT_d}{X}}$,
\quad
and
\quad
$\I_{\le d}\of{X}\defeq\I_{\FT_d}\of{X}$.}
\end{equation*}

\begin{remark}%
\label{remark:frameworks}
A framework for the decomposition of sums of characters of commutative monoids
has been proposed in
Dress-Grabmeier~%
\cite{DG91}
and
derivations for sums of eigenfunctions
\trb{or more generally eigenvectors}
of linear operators
have been developed in
Grigoriev-Karpinski-Singer~%
\cite{GKS91}
and
Peter-Plonka~%
\cite{PP13}.
We recast these frameworks
in the language of Prony structures in
Section~\ref{section:applications}.
See
Remark~\ref{remark:framework-inclusions}
for a diagrammatic overview.

While there is considerable overlap with the one proposed here,
the two approaches make different compromises
between generality and effectivity.
We aim at a formalization of the most general situation in which
Prony's strategy still works.
Our treatment is axiomatic
rather than the explicit constructions of~%
\cite{DG91,GKS91,PP13}.
While trading in some directness,
this abstraction also allows to stay within the language of linear algebra.
When dealing with applications,
a detour through character sums can seem unnatural
\trb{or,
as in the Chebyshev decomposition,
impossible}
given the concrete situation.
In this sense,
we also find our framework to be more effectively verifiable.
\end{remark}

\begin{remark}%
\label{remark:trivial-prony-structure}
For $f\in V$
let~$P_d\of{f}$~denote the matrix
of~$\ev{\le d}{\suppvar_u\of{f}}$~\wrt~the monomial basis
of~$K\ringad{\X}_{\le d}$
and the canonical basis
of~$K^{\suppvar_u\of{f}}$.
Then~$\rb{P_d\of{f}}_{d\in\nat}$~is a Prony structure for~$f$,
\confer~Lemma~\ref{lemma:vanishing-degree-existence}.

For practical computation of the support of~$f$ this Prony structure is useless,
since clearly
$P_d\of{f}$
is the Vandermonde-like matrix
$\Vandermonde_{\FT_d}^{\suppvar_u\of{f}}=\rb{x^\alpha}_{x\in\suppvar_u\of{f},\abs{\alpha}\le d}$
and knowing these
matrices
immediately implies knowledge of the $u$-support of~$f$.
This observation does however provide a possible strategy to construct
Prony structures that may be obtained from some available data,
see
Corollary~\ref{corollary:monomorphism-composed-with-evaluation}.
\end{remark}

We recall the classic
Prony's method for reconstructing univariate exponential sums,
which dates back to~1795~%
\cite{Ric95}.
It is the fundamental example of a Prony structure.

\begin{example}%
\label{example:classic-prony}
For $b\in\complex$
we call the function
\begin{equation*}
\exp_b\colon\nat\to\complex\text{,}
\quad
\alpha\mapsto b^\alpha\text{,}
\end{equation*}
\emph{exponential \trb{with base~$b$}}
and we call
$\complex$-linear combinations of exponentials
\emph{exponential sums}.
Here it is understood that~$0^0=1$.
We denote by
$B\defeq\pset{\exp_b}{b\in\complex}$
the set of all exponentials,
which is a $\complex$-basis of the vector space
\begin{equation*}
V
\defeq
{\Exp}
\defeq
\mgen{B}_\complex
=
\pset{f\colon\nat\to\complex}{\text{$f$~exponential sum}}\text{.}
\end{equation*}
Then the classic Prony problem is to determine
the coefficients~$f_i\in\complex\setminus\set{0}$ and the bases~$b_i\in\complex$
of a given exponential sum~$f=\sum_{i=1}^rf_i\exp_{b_i}\in{\Exp}$.
Of course,
the function
\begin{equation*}
u\colon B\to\complex\text{,}
\quad
\exp_b\mapsto b=\exp_b\of{1}\text{,}
\end{equation*}
is a bijection.
For
an exponential sum $f\in{\Exp}$
and $d\in\nat$,
consider the Hankel matrix
\begin{equation*}
\Hankel_d\of{f}
\defeq
\rb{f\of{\alpha+\beta}}_{\substack{\alpha=0,\dots,d-1\\\beta=0,\dots,d}}
\mathrel{\hphantom{\mathop:}{=}}
\begin{pmatrix}
f\of{0}    &  f\of{1}  &  \cdots  &  f\of{d}     \\
f\of{1}    &  f\of{2}  &  \cdots  &  f\of{d+1}   \\
\vdots     &  \vdots   &  \vdots  &  \vdots      \\
f\of{d-1}  &  f\of{d}  &  \cdots  &  f\of{2d-1}
\end{pmatrix}
\in
\complex^{d\times\rb{d+1}}\text{.}
\end{equation*}
Prony has shown
in his~1795~\emph{Essai}~%
\cite{Ric95}
that
$\Hankel$~is a Prony structure on~$\Exp$
and,
moreover,
for every $f\in{\Exp}$,
$\Pronyindex_{\Hankel}\of{f}=\basisrank_B\of{f}$.
This provides a method to compute
$\suppvar_u\of{f}=\ZL\of{\ker\Hankel_d\of{f}}\subseteq\complex$,
under the assumption that an upper bound $d=d_f\in\nat$ of $\basisrank_B\of{f}$ is known.
\trb{Multivariate} generalizations and variants of Prony's method
will be discussed in
Sections~\ref{section:exponential-sums}
and~\ref{section:applications}
\trb{see also
Peter-Plonka~%
\cite{PP13},
Kunis-Peter-R\"omer-von~der~Ohe~%
\cite{KPRv16},
Sauer~%
\cite{Sau17},
and
Mourrain~%
\cite{Mou18}}.
\end{example}

\begin{remark}%
\label{remark:quasi-prony-structures}
One might be tempted to remove
the technical
\emph{vanishing space condition}
\begin{equation*}
\I_{\FJ_d}\of{\suppvar_u\of{f}}\subseteq\ker\of{P_d\of{f}}
\end{equation*}
from
Definition~\ref{definition:prony-structure}.
For the sake of discussion,
call~$P$~a
\emph{quasi Prony structure for~$f$}
if~$P$~satisfies all the conditions
of a Prony structure for~$f$
in
Definition~\ref{definition:prony-structure}
with the only possible exception of
the vanishing space condition.
We observe the following:
\begin{tlist}
\item
All practically relevant examples of quasi Prony structures
that
we are aware of
are indeed Prony structures.
\item
One of the main reasons why we include the vanishing space condition
in the definition of Prony structures
is that the analogues of
several of our statements on Prony structures do not hold
or are not known to hold for quasi Prony structures;
see,
for example,
Theorem~\ref{theorem:prony-structure-characterization}
and
Theorem~\ref{theorem:relative-prony-structures-construction}.
\item
An \qquot{artificial} example of a quasi Prony structure that is not a Prony structure:
For $d\in\nat$ let
$\FI_d=\set{0,1}$,
$\FJ_d=\set{0,1,2}$,
and
$P_d
\defeq
\begin{pmatrix}
1  &  0  &  0  \\
0  &  1  &  0
\end{pmatrix}
\in
\complex^{\FI_d\times\FJ_d}$.
Then
$\ker\of{P_d}=\mgen{\X^2}_\complex\subseteq\complex\ringad{\X}$,
so $\ZL\of{\ker\of{P_d}}=\ZL\of{\X^2}=\set{0}$,
hence $P=\rb{P_d}_{d\in\nat}$ is a quasi Prony structure for~$f\defeq\exp_0\in{\Exp}$
\trb{\confer~Example~\ref{example:classic-prony}}.
Since $\X\in\I_{\FJ_d}\of{0}\setminus\ker\of{P_d}$ for all~$d$,
$P$~is not a Prony structure for~$f$.
\end{tlist}
\end{remark}

\begin{remark}%
\label{remark:radical-ideals}
~
\begin{tlist}
\item
The generalization of Prony's problem to
polynomial-exponential sums
\trb{sums of functions
$\alpha\mapsto p\of{\alpha}\exp_b\of{\alpha}$
with polynomials~$p$},
also known as \qquot{multiplicity case},
can be found in the univariate case
in
Henrici~%
\cite[Theorem~\textup{7.2}\,\textup{c}]{Hen74}.
Further developments such as
a characterization of sequences that allow interpolation
by polynomial-exponential sums
have been obtained by~%
Sidi~%
\cite{Sid82}
and
a variant based on an associated generalized eigenvalue problem
is given in~%
Lee~%
\cite{Lee07},
see also
Peter-Plonka~%
\cite[Theorem~\textup{2.4}]{PP13}
and
Stampfer-Plonka~%
\cite{SP20+}.
For generalizations of many of these results
to the multivariate setting see
Mourrain~%
\cite{Mou18}.
It would be interesting to extend
the notion of Prony structures to also include these cases.
We leave this for future work.
See also
Remark~\ref{remark:multiplicities}.
\item%
\label{remark:radical-ideals:algebraically-closed-field}
In general,
if~$P\of{f}$~is a Prony structure for~$f$
and~$K$~is algebraically closed,
then,
for all $d\ge\Pronyindex_P\of{f}$,
we have
$\rad\of{\igen{\ker P_d\of{f}}}
=
\I\of{\ZL\of{\ker P_d\of{f}}}
=
\I\of{\suppvar_u\of{f}}$
by Hilbert's Nullstellensatz.
It is an interesting problem whether always or under which conditions the ideal
$\igen{\ker P_d\of{f}}$
is already a radical ideal.
We return to this question in
Section~\ref{section:evaluation-map-and-moeller}
where we provide partial answers
also over not necessarily algebraically closed fields.
\end{tlist}
\end{remark}

\section{Prony structures and the evaluation map}%
\label{section:evaluation-map}

In this section we recall
some well-known properties of
evaluation maps on vector spaces of polynomials
and their kernels.
Since they are the vector spaces of polynomials vanishing on a set~$X\subseteq K^n$,
these kernels play a crucial role
in the theory and application of Prony structures,
which will be made precise in
Theorem~\ref{theorem:prony-structure-characterization}.

We provide in this section the essential facts.
Related issues will be studied in more detail in
Section~\ref{section:evaluation-map-and-moeller}.

\begin{lemma}%
\label{lemma:vanishing-degree-existence}
Let $X\subseteq K^n$.
Then
there is a $d\in\nat$
with
$\igen{\I_{\le d}\of{X}}=\I\of{X}$.
For finite~$X$~this implies
$\ZL\of{\I_{\le d}\of{X}}=X$.
\end{lemma}

\begin{proof}
This
follows
immediately
from the fact
that $S=K\ringad{\X_1,\dots,\X_n}$ is Noetherian
and thus~$\I\of{X}$~is finitely generated for $X\subseteq K^n$.
If~$X$~is finite, then it is Zariski closed.
\end{proof}

\begin{corollary}%
\label{corollary:monomorphism-composed-with-evaluation}
Let $X\subseteq K^n$ be finite.
Then for any $K$-vector space~$W$
and injective $K$-linear map $i\colon K^X\emb W$
one has
$\I_{\le d}\of{X}=\ker\of{i\comp{\ev{\le d}{X}}}$.
In particular,
$\ZL\of{\ker\of{i\comp{\ev{\le d}{X}}}}=X$
for all large~$d$.
The following diagram illustrates the situation.
\begin{equation*}
\begin{tikzcd}
|[alias=Sd]|  S_{\le d}  &&&  |[alias=KX]|  K^X  \\
                         &&&  |[alias=W ]|  W
\ar[  "\ev{\le d}{X}"           ,  from=Sd  ,  to=KX  ,        ]
\ar[  "i"                       ,  from=KX  ,  to=W   ,  hook  ]
\ar[  "i\comp{\ev{\le d}{X}}"'  ,  from=Sd  ,  to=W   ,        ]
\end{tikzcd}
\end{equation*}
\end{corollary}

\begin{proof}
The first statement clearly holds
and the second one follows from
Lemma~\ref{lemma:vanishing-degree-existence}.
\end{proof}

The following
result
on polynomial interpolation
is well-known.

\begin{lemma}%
\label{lemma:polynomial-interpolation}
Let $X\subseteq K^n$ be finite.
If $d\in\nat$
and
$d\ge\card{X}-1$
then
$\ev{\le d}{X}$
is surjective.
\end{lemma}

\begin{proof}
It is easy to see
that given $x\in X$,
there is a polynomial $p\in S$
of degree~$\card{X}-1$
such that $p\of{x}=1$ and $p\of{y}=0$ for $y\in X\setminus\set{x}$
\trb{see,~\eg,~the proof of
Cox-Little-O'Shea~%
\cite[Chapter~\textup{5},
\S\,\textup{3},
Proposition~\textup{7}]{CLO15}}.
By linearity this concludes the proof.
\end{proof}

As the main result of this section
we obtain the following characterization
of Prony structures.

\begin{theorem}%
\label{theorem:prony-structure-characterization}
Given the setup of
Definition~\ref{definition:u-support},
let
$f\in V$,
$B$~an~$F$-basis of~$V$,
$u\colon B\to K^n$ injective,
$\FI$~a sequence of finite sets,
and
$\FJ$~a sequence of finite subsets of~$\nat^n$
with
$\FJ_d\subseteq\FJ_{d+1}$
for all large~$d$
and
$\bigcup_{d\in\nat}\FJ_d=\nat^n$.
Let
$Q\in\prod_{d\in\nat}K^{\FI_d\times\FJ_d}$.
Then the following are equivalent:
\begin{ifflist}
\item%
\label{theorem:prony-structure-characterization:prony-structure}
$Q$~is a Prony structure for~$f$;
\item%
\label{theorem:prony-structure-characterization:monomorphism}
For all large~$d$
there is
an injective $K$-linear map $\eta_d\colon K^{\suppvar_u\of{f}}\emb K^{\FI_d}$
such that the diagram
\begin{equation*}
\begin{tikzcd}
|[alias=KJd]|  K^{\FJ_d}  &&&  |[alias=KId  ]|  K^{\FI_d}             \\
|[alias=SJd]|  S_{\FJ_d}  &&&  |[alias=Ksupp]|  K^{\suppvar_u\of{f}}
\ar[  "Q_d"                           ,  from=KJd    ,  to=KId    ,          ,        ]
\ar[  "\isom"                         ,  from=KJd    ,  to=SJd    ,          ,        ]
\ar[  "\ev{\FJ_d}{\suppvar_u\of{f}}"  ,  from=SJd    ,  to=Ksupp  ,          ,        ]
\ar[  "\eta_d"'                       ,  from=Ksupp  ,  to=KId    ,  dashed  ,  hook  ]
\end{tikzcd}
\end{equation*}
is commutative;
\item%
\label{theorem:prony-structure-characterization:kernel}
For all large~$d$
we have
$\ker\of{Q_d}=\I_{\FJ_d}\of{\suppvar_u\of{f}}$.
\end{ifflist}
\end{theorem}

\begin{proof}
\claimmimpl{\ref{theorem:prony-structure-characterization:prony-structure}}{\ref{theorem:prony-structure-characterization:monomorphism}}
By
Definition~\ref{definition:evaluation-map}
and
since~$Q$~is a Prony structure for~$f$,
for all large~$d$
we have
\begin{equation*}
\ker\of{\ev{\FJ_d}{\suppvar_u\of{f}}}
=
\I_{\FJ_d}\of{\suppvar_u\of{f}}
\subseteq
\ker\of{Q_d}\text{.}
\end{equation*}
By the hypotheses on~$\FJ$,
$\FT_{\card{\suppvar_u\of{f}}}\subseteq\FJ_d$
for all large~$d$.
Then
$\ev{\FJ_d}{\suppvar_u\of{f}}$
is surjective
by
Lemma~\ref{lemma:polynomial-interpolation}.
Together,
these facts imply the existence of
$K$-linear maps~$\eta_d$
such that the required diagrams are commutative.

It remains to show that~$\eta_d$~is injective for all large~$d$.
Let $c\in\nat$ be such that
for all $d\ge c$
we have that
\begin{equation*}
\text{$\ZL\of{\ker Q_d}=\suppvar_u\of{f}$,
\quad
$\ev{\FJ_d}{\suppvar_u\of{f}}$ is surjective,
\quad
and~$\eta_d$~exists.}
\end{equation*}
Let $v\in\ker\of{\eta_d}$.
By surjectivity of~$\ev{\FJ_d}{\suppvar_u\of{f}}$
we have
$\ev{\FJ_d}{\suppvar_u\of{f}}\of{p}=v$
for some $p\in S_{\FJ_d}$.
Then
$Q_d\of{p}
=
\eta_d\of{\ev{\FJ_d}{\suppvar_u\of{f}}\of{p}}
=
\eta_d\of{v}
=
0$.
Thus,
we have
\begin{equation*}
p
\in
\ker\of{Q_d}
\subseteq
\I\of{\ZL\of{\ker Q_d}}
=
\I\of{\suppvar_u\of{f}}
=
\ker\of{\ev{}{\suppvar_u\of{f}}}\text{.}
\end{equation*}
Hence,
$v=\ev{}{\suppvar_u\of{f}}\of{p}=0$.
Thus,~$\eta_d$~is injective.

\claimmimpl{\ref{theorem:prony-structure-characterization:monomorphism}}{\ref{theorem:prony-structure-characterization:kernel}}
Since~$\eta_d$~exists and is injective \trb{for all large~$d$},
we have
\begin{equation*}
\ker\of{Q_d}
=
\ker\of{\eta_d\comp\ev{\FJ_d}{\suppvar_u\of{f}}}
=
\ker\of{\ev{\FJ_d}{\suppvar_u\of{f}}}
=
\I_{\FJ_d}\of{\suppvar_u\of{f}}\text{.}
\end{equation*}

\claimmimpl{\ref{theorem:prony-structure-characterization:kernel}}{\ref{theorem:prony-structure-characterization:prony-structure}}
By our
hypothesis
and
Lemma~\ref{lemma:vanishing-degree-existence},
for all large~$d$
we have
\begin{equation*}
\ZL\of{\ker Q_d}
=
\ZL\of{\I_{\FJ_d}\of{\suppvar_u\of{f}}}
=
\ZL\of{\I\of{\suppvar_u\of{f}}}
=
\suppvar_u\of{f}\text{.}
\end{equation*}
The vanishing space condition
in
Definition~\ref{definition:prony-structure}\,\eqref{equation:definition:prony-structure}
is obviously satisfied.
\end{proof}

The art of constructing a
\qquot{computable}
Prony structure
for a given
$f\in V$
and the very heart of Prony's method
is to find an injective $K$-linear map $\eta_d\colon K^{\suppvar_u\of{f}}\emb W_d$
into a $K$-vector space~$W_d$
such that
\trb{a matrix of}
the composition
\begin{equation*}
P_d\of{f}\defeq\eta_d\comp{\ev{\le d}{\suppvar_u\of{f}}}\colon S_{\le d}\to W_d
\end{equation*}
can be computed from standard data of~$f$.

\begin{remark}%
\label{remark:prony-structure-for-generating-systems}
Let~$B$~be a generating subset of~$V$.
One can formulate a variation of
Theorem~\ref{theorem:prony-structure-characterization}
insofar that if one of the
conditions~\ref{theorem:prony-structure-characterization:monomorphism}
or~\ref{theorem:prony-structure-characterization:kernel}
holds
for all $f\in V$
and all representations
$f=\sum_{b\in M}f_bb$ with $M\subseteq B$ finite and $f_b\in F\setminus\set{0}$,
and replacing each occurrence of~$\suppvar_u\of{f}$ by~$u\of{M}$,
then~$B$~is a basis of~$V$~and~$Q_d\of{f}$~induces a Prony structure on~$V$.
Indeed,
$\ZL\of{\ker Q_d\of{f}}=M$
implies that~$M$~is uniquely determined by~$Q_d\of{f}$,
which implies the desired conclusion.
\end{remark}

The following
Proposition~\ref{proposition:vanishing-degree-construction}\,\ref{proposition:vanishing-degree-construction:cardinality-works}
is
a version of
Lemma~\ref{lemma:vanishing-degree-existence}
that provides the upper bound $d=\card{X}$ for the
\qquot{stabilization index}
of the ascending sequence of ideals $\rb{\mgen{\I_{\le d}\of{X}}}_{d\in\nat}$.
In part~\ref{proposition:vanishing-degree-construction:cardinality-sharp}
it is shown
that $\card{X}-1$ is not in general an upper bound.

\begin{proposition}%
\label{proposition:vanishing-degree-construction}
The following holds:
\begin{tlist}
\item%
\label{proposition:vanishing-degree-construction:cardinality-works}
Let
$X\subseteq K^n$
be finite.
With $d\defeq\card{X}$
we have
\begin{equation*}
\igen{\I_{\le d}\of{X}}=\I\of{X}\text{.}
\end{equation*}
\item%
\label{proposition:vanishing-degree-construction:cardinality-sharp}
Let $K$~be an infinite field.
Then
for every $d\in\nat$
there is an $X\subseteq K^n$
with $\card{X}=d+1$
such that
$\igen{\I_{\le d}\of{X}}\subsetneqq\I\of{X}$.
\end{tlist}
\end{proposition}

\begin{proof}
\ref{proposition:vanishing-degree-construction:cardinality-works}
This is part of the proof of
Kunis-Peter-R\"omer-von~der~Ohe~%
\cite[Theorem~\textup{3.1}]{KPRv16}.

\ref{proposition:vanishing-degree-construction:cardinality-sharp}
Let $d\in\nat$.
Since~$K$~is infinite,
there exists
\begin{equation*}
\text{$X=\pset{\rb{x,0,\dots,0}\in K^n}{x\in X_1}$
for some $X_1\subseteq K$
with
$\card{X}=\card{X_1}=d+1$.}
\end{equation*}
Let $\I\of{X}=\igen{E}$
for some $E\subseteq S$.
We claim that there is a $p\in E$ with $\deg\of{p}>d$.

For $p\in S$,
let $\widetilde{p}\defeq p\of{\X_1,0,\dots,0}$.
Assume $\widetilde{p}=0$ for all $p\in E$.
For $y\in K$ we have
$\rb{y,0,\dots,0}\in\ZL\of{E}=\ZL\of{\I\of{X}}=X$
and hence
$y\in X_1$.
We get the contradiction $X_1=K$.

Thus there is a $p\in E$ with $\widetilde{p}\ne0$.
Since $\widetilde{p}\of{x}=0$ for $x\in X_1$
and $\card{X_1}=d+1$,
we have
\begin{equation*}
\deg\of{p}\ge\deg\of{\widetilde{p}}\ge d+1\text{.}
\end{equation*}
This concludes the proof.
\end{proof}

\section{Properties of the evaluation map and a theorem of M\"oller}%
\label{section:evaluation-map-and-moeller}

Continuing the discussion in
Section~\ref{section:evaluation-map}
we study in the following further properties of evaluation maps
and we provide
partial answers to the question raised in
Remark~\ref{remark:radical-ideals}\,\ref{remark:radical-ideals:algebraically-closed-field}.

This section is to some degree independent from the rest of the article.
The reader who wishes to continue directly with applications of Prony structures
and is not particularly concerned with
the ideal-theoretic issues treated here
can safely skip this section.
The consequences of the results of this section for Prony structures
are summarized in
Corollary~\ref{corollary:moeller-prony-structures}.

We are grateful towards
H.\,M.~M\"oller
for inspiring discussions
related to these results,
in particular for allowing us to
include
Theorem~\ref{theorem:moeller-ideal-basis}
and its proof~%
\cite{Moe17}.

As before,
let $S=K\ringad{\X_1,\dots,\X_n}$
be the polynomial ring in~$n$~indeterminates over the field~$K$.
In the following
we do not distinguish
between $\alpha\in\nat^n$
and the monomial $\X^\alpha\in\Mon\of{S}$.
For general facts about initial ideals and Gr\"obner bases
see,~\eg,~%
Cox-Little-O'Shea~%
\cite{CLO15}.

\begin{remark}
Let $X\subseteq K^n$ be finite.
A direct consequence of
Proposition~\ref{proposition:vanishing-degree-construction}\,\ref{proposition:vanishing-degree-construction:cardinality-works}
is that
for all $d\ge\card{X}$
the vanishing spaces
$\I_{\FT_d}\of{X}$
generate the same radical ideal in~$S$
\trb{namely,~$\I\of{X}$}.
\end{remark}

As a consequence we get immediately:

\begin{corollary}%
\label{corollary:asymptotically-radical}
Given the setup of
Definition~\ref{definition:prony-structure},
let~$P\of{f}$~be a Prony structure for $f\in V$
with $\FJ_d\subseteq\FJ_{d+1}$ for all large~$d$
and $\bigcup_{d\in\nat}\FJ_d=\nat^n$.
Then
for all large~$d$
\begin{equation*}
\igen{\ker P_d\of{f}}
=
\I\of{\suppvar_u\of{f}}\text{.}
\end{equation*}
In particular,
for all large~$d$,
$\igen{\ker P_d\of{f}}$~is a radical ideal in~$S$.
\end{corollary}

\begin{proof}
Let
$X\defeq\suppvar_u\of{f}$
and
$r\defeq\basisrank\of{f}=\card{X}$.
For all large~$d$
we have
$\FT_r\subseteq\FJ_d$
and
$\ker P_d\of{f}
=
\I_{\FJ_d}\of{X}
\supseteq
\I_{\FT_r}\of{X}$.
Since also $\FJ_d\subseteq\FT_e$ for an $e\in\nat$,
we have
\begin{equation*}
\I\of{X}
=
\igen{\I_{\FT_r}\of{X}}
\subseteq
\igen{\I_{\FJ_d}\of{X}}
\subseteq
\igen{\I_{\FT_e}\of{X}}
\subseteq
\I\of{X}\text{.}
\end{equation*}
This concludes the proof.
\end{proof}

Observe that~$\igen{\I_D\of{X}}_S$~is not a radical ideal in general.
This is shown already by the example
$n=1$,
$X=\set{0}$,
$D=\set{\X_1^2}$,
where
$\igen{\I_D\of{X}}=\igen{\X_1^2}_S$.
Note that for a given
$X\subseteq K^n$,
the map
$\ev{\le d}{X}$
can be surjective
also for $d<\card{X}-1$.
Furthermore,
$\I_{\le d}\of{X}$
could also generate a radical ideal for small~$d$.
The following simple example illustrates this.

\begin{example}%
\label{example:evaluation-surjectivity-and-radical}
Let
$X\defeq\set{\rb{0,0},\rb{1,0},\rb{0,1}}\subseteq K^2$.
One can see immediately that~$\ev{\le1}{X}$~is bijective
by considering its matrix
\begin{equation*}
\Vandermonde_{\le1}^X
=
\rb{t\of{x}}_{\substack{x\in X\\t\in\FT_1}}
=
\begin{pmatrix}
1  &  0  &  0  \\
1  &  1  &  0  \\
1  &  0  &  1
\end{pmatrix}
\in
K^{X\times\FT_1}\text{.}
\end{equation*}
Therefore,
$\ev{\le1}{X}$~is surjective
and
$\I_{\le1}\of{X}=\ker\of{\ev{\le1}{X}}=\set{0}$.
So $\igen{\I_{\le1}\of{X}}$ is the zero ideal of~$S$,
which is prime and thus radical
\trb{and of course not equal to~$\I\of{X}$}.
The vanishing ideal~$\I\of{X}$ of~$X$
is generated by
\begin{equation*}
\ker\of{\ev{\le2}{X}}
=
\mgen{\X_1\rb{\X_1-1},\X_2\rb{\X_2-1},\X_1\X_2}_K\text{.}
\end{equation*}
\end{example}

Having these facts in mind
we
consider special situations and
prove
results
related to
Corollary~\ref{corollary:asymptotically-radical}
and
Example~\ref{example:evaluation-surjectivity-and-radical}.

For a monomial order~$<$ on~$\Mon\of{S}$
and an ideal~$I$~of~$S$
we denote by
\begin{equation*}
\normalset_<\of{I}
\defeq
\Mon\of{S}\setminus\initial_<\of{I}
\end{equation*}
the
\emph{normal set of~$I$}.
From now on we omit the monomial order from the notation
and write,~\eg,~$\initial\of{I}$ and $\normalset\of{I}$
for~$\initial_<\of{I}$ and~$\normalset_<\of{I}$,
respectively.

For example,
for $I=\I\of{X}$ with $X\subseteq K^2$ as in
Example~\ref{example:evaluation-surjectivity-and-radical},
one has
\begin{equation*}
\text{$\initial\of{I}=\igen{\X_1^2,\X_1\X_2,\X_2^2}$
and thus
$\normalset\of{I}=\set{1,\X_1,\X_2}$}
\end{equation*}
for the degree reverse lexicographic order~$<$.

\begin{lemma}%
\label{lemma:moeller-basis-choice}
Let~$<$~be a monomial order on~$\Mon\of{S}$,
$X\subseteq K^n$ be finite
and
$I\defeq\I\of{X}$.
Then the following holds:
\begin{tlist}
\item%
\label{lemma:moeller-basis-choice:normal-set-bijection}
${\ev{\normalset\of{I}}{X}}\colon S_{\normalset\of{I}}\to K^X$
is bijective.
In particular,
$\card{\normalset\of{I}}=\card{X}$.
\item%
\label{lemma:moeller-basis-choice:choice}
Let $D\subseteq\Mon\of{S}$
be such that~${\ev{D}{X}}\colon S_D\to K^X$~is surjective.
Then there is a $C\subseteq\Mon\of{S}$
with the following properties:
\begin{tlist2}
\item%
\label{lemma:moeller-basis-choice:choice:subset}
$C\subseteq D$.
\item%
\label{lemma:moeller-basis-choice:choice:bijection}
${\ev{C}{X}}\colon S_C\to K^X$~is bijective.
In particular,
$\card{C}=\card{X}=\card{\normalset\of{I}}$.
\item%
\label{lemma:moeller-basis-choice:choice:order}
For all $t\in D\setminus C$
we have
$\ev{D}{X}\of{t}\in\pmgen{\ev{C}{X}\of{s}}{\text{$s\in C$ and $s<t$}}_K$.
\end{tlist2}
\end{tlist}
\end{lemma}

\begin{proof}
\ref{lemma:moeller-basis-choice:normal-set-bijection}
It is a standard fact that
$S_{\normalset\of{I}}\isom S/I\isom K^X$,
see,
for example,
Cox-Little-O'Shea~%
\cite[Chapter~\textup{5},
\S\,\textup{3},
Proposition~\textup{4}]{CLO15}.
Let $p\in\ker\of{\ev{\normalset\of{I}}{X}}$ and suppose that $p\ne0$.
Then $\initial\of{p}\in\initial\of{I}\cap\normalset\of{I}=\emptyset$,
a contradiction.
Thus,
$\ev{\normalset\of{I}}{X}$~is injective
and hence an isomorphism.

\ref{lemma:moeller-basis-choice:choice}
Note that necessarily $\card{D}\ge\card{X}$.
We prove the assertion by induction on $k=\card{D}-\card{X}\in\nat$.
If $k=0$,
then $\card{D}=\card{X}$.
So~$\ev{D}{X}$~is bijective
and $C=D$ works trivially.

Let $k\ge1$.
Then $\card{D}>\card{X}$
and
the elements~$\ev{D}{X}\of{t}$, $t\in D$,
are linearly dependent in~$K^X$.
Hence there are $\lambda_t\in K$ with
$\sum_{t\in D}\lambda_t\ev{D}{X}\of{t}=0$
and $\lambda_t\ne0$ for at least one~$t\in D$.
Let
\begin{equation*}
\text{$t_0\defeq\max\nolimits_<\pset{t\in D}{\lambda_t\ne0}$
and
$D_1\defeq D\setminus\set{t_0}$.}
\end{equation*}
Clearly,
${\ev{D_1}{X}}\colon S_{D_1}\to K^X$
is surjective
and $\card{D_1}-\card{X}=k-1$.
By induction hypothesis
there is a $C_1\subseteq D_1$
such that
\begin{equation*}
\text{${\ev{C_1}{X}}\colon S_{C_1}\to K^X$
is bijective
and
$\ev{D_1}{X}\of{t}\in\pmgen{\ev{C_1}{X}\of{s}}{\text{$s\in C_1$, $s<t$}}_K$
for all $t\in D_1\setminus C_1$.}
\end{equation*}
Clearly,
$C_1\subseteq D$.
We claim that~$C\defeq C_1$~fulfills the assertion
also for~$D$.
It remains to show
statement~\ref{lemma:moeller-basis-choice:choice:order}
for $t=t_0$.
For this let
$U\defeq\pmgen{\ev{C}{X}\of{s}}{\text{$s\in C$, $s<t_0$}}_K$.
From the linear dependency above
it follows that
\begin{equation*}
\text{$\ev{D}{X}\of{t_0}
=
\sum_{s\in D_1}\mu_s\ev{D_1}{X}\of{s}
=
\sum_{s\in C}\mu_s\ev{C}{X}\of{s}
+
\sum_{s\in D_1\setminus C}\mu_s\ev{D_1}{X}\of{s}$
with $\mu_s\in K$.}
\end{equation*}
Trivially $\sum_{s\in C}\mu_s\ev{C}{X}\of{s}\in U$
since
by the choice of~$t_0$
we have $s<t_0$ for all~$s\in D$ with $\mu_s\ne0$.
Also by the choice of~$t_0$
and the induction hypothesis mentioned above
we have
$\sum_{s\in D_1\setminus C}\mu_s\ev{D_1}{X}\of{s}
\in
U$.
Thus we have $\ev{D}{X}\of{t_0}\in U$.
This concludes the proof.
\end{proof}

\begin{remark}
Let the notation be as in
Lemma~\ref{lemma:moeller-basis-choice}\,\ref{lemma:moeller-basis-choice:choice}
and~$\ev{D}{X}$~surjective.
There are the following interesting questions:
\begin{qlist}
\item
Under which conditions
do we have
$\normalset\of{I}\subseteq D$?
\item
Under which conditions
does
$C=\normalset\of{I}$
satisfy~\ref{lemma:moeller-basis-choice:choice:subset},
\ref{lemma:moeller-basis-choice:choice:bijection},
and~\ref{lemma:moeller-basis-choice:choice:order}
in
Lemma~\ref{lemma:moeller-basis-choice}\,\ref{lemma:moeller-basis-choice:choice}?
\end{qlist}
Of course,
$C=\normalset\of{I}$
implies that
$\normalset\of{I}\subseteq D$.
A simple example that shows
$\normalset\of{I}\subseteq D$
does not hold in general is given
by
$n=1$,
$X=\set{1}\subseteq K$,
$D=\set{\X_1}\subseteq\Mon\of{S}$.
\end{remark}

\begin{definition}
Let~$<$~be a monomial order on~$\Mon\of{S}$
and $D\subseteq\Mon\of{S}$
be an order ideal~\wrt~divisibility.
We call~$D$~%
\emph{distinguished}
if for all $t\in D$ and $s\in\Mon\of{S}\setminus D$
we have $t<s$.

For an arbitrary non-empty order ideal~$D\subseteq\Mon\of{S}$
we define
\begin{equation*}
\border\of{D}\defeq\rb{\X_1D\cup\dots\cup\X_nD}\setminus D\text{.}
\end{equation*}
We also set
\begin{equation*}
\border\of{\emptyset}\defeq\set{1}\text{.}
\end{equation*}
Usually,
$\border\of{D}$
is called
the
\emph{border of~$D$}.
\end{definition}

\begin{example}
Our standard example and a counterexample related to distinguished order ideals are the following.
\begin{tlist}
\item
Let
$d\in\nat$.
Then
\begin{equation*}
D
\defeq
\FT_d
=
\pset{\alpha\in\nat^n}{\alpha_1+\dots+\alpha_n\le d}
\end{equation*}
is a distinguished order ideal~\wrt~$<_\degrevlex$
\trb{or any other degree compatible monomial order}.
\item
Clearly,
for any $n\in\nat$,
\begin{equation*}
D
\defeq
\FM_d
=
\pset{\alpha\in\nat^n}{\max\set{\alpha_1,\dots,\alpha_n}\le d}
\end{equation*}
is an order ideal.
For $n\ge2$ and $d\ge1$,
there is no monomial order~$<$ on~$\Mon\of{S}$
such that~$D$~is a distinguished order ideal~\wrt~$<$.
Indeed,
if
$\X_2>\X_1$,
then
$D\ni\X_2\X_1^d>\X_1\X_1^d=\X_1^{d+1}\notin D$.

It would be interesting to extend the results of this
section
to more general
settings.
Since this is outside the scope
of this article,
we omit this discussion here.
\end{tlist}
\end{example}

\begin{lemma}%
\label{lemma:distinguished-order-ideal}
Let~$<$~be a monomial order on~$\Mon\of{S}$,
$X\subseteq K^n$ be finite
and
$D\subseteq\Mon\of{S}$
be a distinguished order ideal~\wrt~$<$
such that~$\ev{D}{X}$~is surjective.
Let $I\defeq\I\of{X}$
and $C\subseteq D$ be as in
Lemma~\ref{lemma:moeller-basis-choice}\,\ref{lemma:moeller-basis-choice:choice}.
For $t\in\Mon\of{S}$
let $p_t\in S_C$
be the uniquely determined polynomial such that $\ev{C}{X}\of{p_t}=\ev{}{X}\of{t}$
and set $q_t\defeq t-p_t$.
Then the following holds:
\begin{tlist}
\item%
\label{lemma:distinguished-order-ideal:q_t}
For
$t\in\Mon\of{S}$
we have
$q_t\in I$.
\item%
\label{lemma:distinguished-order-ideal:supp(p_t)}
For
$t\in\Mon\of{S}\setminus C$
we have
$\polysupp\of{p_t}\subseteq\pset{s\in C}{s<t}$.
\item%
\label{lemma:distinguished-order-ideal:in(q_t)}
For
$t\in\Mon\of{S}\setminus C$
we have
$\initial\of{q_t}=t$.
\item%
\label{lemma:distinguished-order-ideal:supp(p)-C}
For
$p\in I\setminus\set{0}$
we have
$\polysupp\of{p}\nsubseteq C$,
\idest~$p\notin S_C$.
\item%
\label{lemma:distinguished-order-ideal:normal-set}
We have
$C=\normalset\of{I}$.
\end{tlist}
Here,
$\polysupp\of{p}$
denotes the support of~$p$~\wrt~the monomial basis of~$S$.
\end{lemma}

\begin{proof}
\ref{lemma:distinguished-order-ideal:q_t}
This is an immediate consequence of the definition,
since
$I=\ker\of{\ev{}{X}}$.

\ref{lemma:distinguished-order-ideal:supp(p_t)}
If $t\in D\setminus C$
then there are $\mu_s\in K$
such that
\begin{equation*}
\ev{D}{X}\of{t}=\sum_{s\in C,s<t}\mu_s\ev{C}{X}\of{s}=\ev{C}{X}\sumof{\sum_{s\in C,s<t}\mu_ss}\text{.}
\end{equation*}
Hence $p_t=\sum_{s\in C,s<t}\mu_ss$,
and clearly $\polysupp\of{p_t}\subseteq\pset{s\in C}{s<t}$.
If $t\in\Mon\of{S}\setminus D$
then $t>s$ for all $s\in D$
since~$D$~is a distinguished order ideal.
In particular,
we see also in this case that
$\polysupp\of{p_t}\subseteq C=\pset{s\in C}{s<t}$,
finishing the proof of the claim.

\ref{lemma:distinguished-order-ideal:in(q_t)}
This is an immediate consequence of
part~\ref{lemma:distinguished-order-ideal:supp(p_t)}.

\ref{lemma:distinguished-order-ideal:supp(p)-C}
Suppose that $\polysupp\of{p}\subseteq C$.
Then $p\in\I_C\of{X}=\ker\of{\ev{C}{X}}=\set{0}$,
a contradiction.

\ref{lemma:distinguished-order-ideal:normal-set}
If $t\in\Mon\of{S}\setminus C$
then $t=\initial\of{q_t}\in\initial\of{I}$
by
part~\ref{lemma:distinguished-order-ideal:in(q_t)}.
Thus $\normalset\of{I}\subseteq C$
and since $\card{\normalset\of{I}}=\card{X}=\card{C}$,
we have $\normalset\of{I}=C$.
\end{proof}

\begin{corollary}
% \label{corollary:distinguished-order-ideal-surjective-normal-set}
Let~$<$~be a monomial order on~$\Mon\of{S}$,
$X\subseteq K^n$ be finite,
$I\defeq\I\of{X}$,
and
$D\subseteq\Mon\of{S}$
be a distinguished order ideal~\wrt~$<$.
Then the following are equivalent:
\begin{ifflist}
\item%
\label{corollary:distinguished-order-ideal-surjective-normal-set:surjective}
$\ev{D}{X}$~is surjective;
\item%
\label{corollary:distinguished-order-ideal-surjective-normal-set:normal-set}
$\normalset\of{I}\subseteq D$.
\end{ifflist}
\end{corollary}

\begin{proof}
\claimmimpl{\ref{corollary:distinguished-order-ideal-surjective-normal-set:surjective}}{\ref{corollary:distinguished-order-ideal-surjective-normal-set:normal-set}}
Let $t\in\normalset\of{I}$
and let $C\subseteq D$ be as in
Lemma~\ref{lemma:moeller-basis-choice}\,\ref{lemma:moeller-basis-choice:choice}.
Then we have
$\normalset\of{I}=C$
by
Lemma~\ref{lemma:distinguished-order-ideal}\,\ref{lemma:distinguished-order-ideal:normal-set}
and thus $\normalset\of{I}\subseteq D$.

\claimmimpl{\ref{corollary:distinguished-order-ideal-surjective-normal-set:normal-set}}{\ref{corollary:distinguished-order-ideal-surjective-normal-set:surjective}}
By
Lemma~\ref{lemma:moeller-basis-choice}\,\ref{lemma:moeller-basis-choice:normal-set-bijection},
$\ev{\normalset\of{I}}{X}$~is bijective,
and since $\normalset\of{I}\subseteq D$,
$\ev{D}{X}$~is surjective.
\end{proof}

The special case of
the next theorem
for a degree compatible monomial order~$<$
and
$D=\FT_d$
can already be found in~%
\cite[Theorem~\textup{2.48}]{von17}.

\begin{theorem}%
[M\"oller]%
\label{theorem:moeller-ideal-basis}
Let~$<$~be a monomial order on~$\Mon\of{S}$,
$X\subseteq K^n$ finite,
and~$D$~a distinguished order ideal~\wrt~$<$
such that~$\ev{D}{X}$~is surjective.
Then there is a
Gr\"obner basis~$G$~of~$\I\of{X}$
such that
\begin{equation*}
\text{$G\subseteq S_{D\cup\border\of{D}}$
and
$\card{G}=\card{D}+\card{\border\of{D}}-\card{X}$.}
\end{equation*}
\end{theorem}

\begin{proof}
Let
$I\defeq\I\of{X}$
and let
$C=\normalset\of{I}\subseteq D$,
$p_t\in S_C$,
and
$q_t=t-p_t$
be as in
Lemma~\ref{lemma:distinguished-order-ideal}.

Define
\begin{equation*}
G
\defeq
\pset{q_s}{s\in D\cup\border\of{D}\setminus C}
\subseteq
I\text{.}
\end{equation*}
We show that~$G$~is a Gr\"obner basis of~$I$.
Set $J\defeq\igen{\initial\of{G}}_S$.
It suffices to show that $J=\initial\of{I}$.
It is clear that
$J\subseteq\initial\of{I}$.
The reverse inclusion is certainly true if $X=\emptyset$,
since then
\begin{equation*}
\text{$I=\igen{1}=\initial\of{I}$,
$C=\emptyset$,
$1\in D\cup\border\of{D}$,
and
$1=\initial\of{q_1}\in\initial\of{G}$.}
\end{equation*}
Thus let~\wlogen~$X\ne\emptyset$.
Assume that
$\initial\of{I}\nsubseteq J$.
Then there is a monomial
$s\in\initial\of{I}\setminus J$.
Let~$t$~be a minimal monomial generator of~$\initial\of{I}$
with $t\divides s$.
Since $t\in\initial\of{I}$
we have $t\notin\normalset\of{I}=C$.

\emph{Case~1:}
$t\in D$.
Then $q_t\in G$
and $t=\initial\of{q_t}\in\initial\of{G}$,
hence $s\in\igen{\initial\of{G}}=J$,
a contradiction.

\emph{Case~2:}
$t\notin D$.
Since $X\ne\emptyset$
we have $t\ne1$,
so there is a $j\in\set{1,\dots,n}$
such that $\X_j\divides t$.
Let $\widetilde{t}\defeq t/\X_j$.
Since~$t$~is a minimal generator of~$\initial\of{I}$,
we have
$\widetilde{t}\notin\initial\of{I}$,
so
$\widetilde{t}\in C\subseteq D$.
Hence,
$t
=
\X_j\widetilde{t}
\in
\rb{\X_jD}\setminus D
\subseteq
\border\of{D}
\subseteq
\initial\of{G}$.
Thus we obtain that
$s\in\igen{\initial\of{G}}=J$,
again a contradiction.

Thus we have $\initial\of{I}\subseteq J$
and~$G$~is a Gr\"obner basis of~$I$.
By
Lemma~\ref{lemma:distinguished-order-ideal}
it is clear that
$\card{G}
=
\card{D\cup\border\of{D}\setminus C}
=
\card{D}+\card{\border\of{D}}-\card{X}$.
Moreover,
for
$t\in D\cup\border\of{D}\setminus C$
we have
$\polysupp\of{q_t}
=
\set{t}\cup\polysupp\of{p_t}
\subseteq
\set{t}\cup\pset{s\in C}{s<t}
\subseteq
D\cup\border\of{D}$,
\idest,~$q_t\in S_{D\cup\border\of{D}}$,
which concludes the proof.
\end{proof}

Note that in
Theorem~\ref{theorem:moeller-ideal-basis},
in general~$G$~contains a border prebasis induced by~$\border\of{D}$.
In particular,
if the distinguished order ideal~$D$~equals~$\normalset\of{I}$,
then~$G$~is a border basis of~$I$.
See,~\eg,~%
Kreuzer-Robbiano~%
\cite[Section~\textup{6.4}]{KR05}
for further details related to the theory of border bases.

We list two immediate consequences of
Theorem~\ref{theorem:moeller-ideal-basis}
in the following corollary.

\begin{corollary}
The following holds:
\begin{tlist}
\item
With the notation and assumptions as in
Theorem~\ref{theorem:moeller-ideal-basis},
$\I_{D\cup\border\of{D}}\of{X}$
generates a radical ideal in~$S$.
\item
If~$\ev{\FT_d}{X}$~is surjective
then
$\I_{\FT_{d+1}}\of{X}$
generates a radical ideal in~$S$.
\end{tlist}
\end{corollary}

We have the following implications for Prony structures.

\begin{corollary}%
\label{corollary:moeller-prony-structures}
Given the setup of
Definition~\ref{definition:prony-structure},
let~$P\of{f}$~be a Prony structure for~$f\in V$.
Let $d\in\nat$
be such that
$\ker P_d\of{f}=\I_{\FJ_d}\of{\suppvar_u\of{f}}$.
If there is a distinguished order ideal~$D$
\trb{\wrt~some monomial order~$<$~on~$\Mon\of{S}$}
such that~$\ev{D}{\suppvar_u\of{f}}$
is surjective
and $D\cup\border\of{D}\subseteq\FJ_d$
then
\begin{equation*}
\igen{\ker P_d\of{f}}=\I\of{X}\text{.}
\end{equation*}
In particular,
$\igen{\ker P_d\of{f}}$~is a radical ideal in~$S$.
\end{corollary}

\section{Prony structures for multivariate exponential sums}%
\label{section:exponential-sums}

In this section we discuss Prony structures
for multivariate exponential sums
based on Hankel-like and Toeplitz-like matrices.
Because for the Toeplitz case
we need evaluations also at negative arguments,
we have to consider two different variants of exponentials.
One has only non-negative arguments
and no restrictions on the bases in~$K^n$.
The other one is defined also for negative \trb{integer} arguments
and the restriction that the bases lie on the algebraic torus~$\rb{K\setminus\set{0}}^n$.
Observe that it is not possible to define
Toeplitz versions of Prony's method for the first variant.

That Prony's methods can be generalized to these settings was shown
in
Kunis-Peter-R\"omer-von~der~Ohe~%
\cite{KPRv16},
Sauer~%
\cite{Sau17},
and
Mourrain~%
\cite{Mou18}.
Here we provide a new perspective on these results.
Prony structures are a common abstraction
of both Hankel and Toeplitz variants of Prony's method.

The following notation generalizes
the univariate case in
Example~\ref{example:classic-prony}.
Here and in the following we write
$\unit_1,\dots,\unit_n$
for the standard basis vectors of~$K^n$.

\begin{definition}%
\label{definition:exponential-sums-with-domain-NN^n}
Let~$K$~be a field
and~$F$~a subfield of~$K$.
\begin{tlist}
\item
For
$b\in K^n$,
let
\begin{equation*}
\exp_b\colon\nat^n\to K\text{,}\quad\alpha\mapsto b^\alpha=\prod_{j=1}^nb_j^{\alpha_j}\text{,}
\end{equation*}
denote the
\emph{\trb{$n$-variate} exponential
with
base~$b$
\trb{with domain~$\nat^n$}}.
For a subset $Y\subseteq K^n$ let
$B_Y\defeq\pset{\exp_b}{b\in Y}$.
We denote the $F$-subvector space of~$K^{\nat^n}$ generated by~$B_Y$
with
\begin{equation*}
\Exp^n_Y\of{F}
\defeq
\mgen{B_Y}_F\text{.}
\end{equation*}
We call the elements of~$\Exp^n_Y\of{F}$
\emph{\trb{$n$-variate}
exponential sums
\trb{with domain~$\nat^n$}}.
Furthermore,
we denote by~$u_Y$
the function
\begin{equation*}
u_Y\colon B_Y\to K^n\text{,}\quad\exp_b\mapsto\rb{\exp_b\of{\unit_1},\dots,\exp_b\of{\unit_n}}=b\text{.}
\end{equation*}
Trivially,
$u_Y$~is injective.
\item
Let~$\FI,\FJ$~be sequences of finite subsets of~$\nat^n$.
For $f\in\Exp^n_Y\of{F}$
and $d\in\nat$
let
\begin{equation*}
\Hankel_d\of{f}
\defeq
\Hankel_{\FI,\FJ,d}\of{f}
\defeq
\rb{f\of{\alpha+\beta}}_{\substack{\alpha\in\FI_d\\\beta\in\FJ_d}}
\in
K^{\FI_d\times\FJ_d}\text{.}
\end{equation*}
\end{tlist}
\end{definition}

We will see in
Theorem~\ref{theorem:exponential-sums}
that~$\Hankel_d\of{f}$~induces a Prony structure on the space of exponential sums~$\Exp^n_Y\of{F}$,
and that therefore the set~$B_Y$~is a basis of~$\Exp^n_Y\of{F}$.

The following is a variation of
Definition~\ref{definition:exponential-sums-with-domain-NN^n}
where all bases~$b$~are restricted to lie on the algebraic torus~$\rb{K\setminus\set{0}}^n$.
This allows also for non-negative arguments,~\idest~%
the exponentials are functions on the domain~$\integ^n$.
As a consequence it is possible to define not only sequences of Hankel-like
but also of Toeplitz-like matrices
associated to an exponential sum \trb{with domain~$\integ^n$}.
In order to avoid any possible confusion,
we write out the definition in full.

\begin{definition}
Let~$K$~be a field
and~$F$~a subfield of~$K$.
\begin{tlist}
\item
For
$b\in\rb{K\setminus\set{0}}^n$,
let
\begin{equation*}
\exp_{\integ,b}\colon\integ^n\to K\text{,}\quad\alpha\mapsto b^\alpha=\prod_{j=1}^nb_j^{\alpha_j}\text{,}
\end{equation*}
denote the
\emph{\trb{$n$-variate} exponential
with
base~$b$
\trb{with domain~$\integ^n$}}.
For a subset $Y\subseteq\rb{K\setminus\set{0}}^n$ let
$B_{\integ,Y}\defeq\pset{\exp_{\integ,b}}{b\in Y}$.
We denote the $F$-subvector space of~$K^{\integ^n}$ generated by~$B_{\integ,Y}$
with
\begin{equation*}
\Exp^n_{\integ,Y}\of{F}
\defeq
\mgen{B_{\integ,Y}}_F\text{.}
\end{equation*}
We call the elements of~$\Exp^n_{\integ,Y}\of{F}$
\emph{\trb{$n$-variate}
exponential sums
\trb{with domain~$\integ^n$}}.
Furthermore,
we denote by~$u_{\integ,Y}$
the function
\begin{equation*}
u_{\integ,Y}\colon B_{\integ,Y}\to K^n\text{,}\quad\exp_{\integ,b}\mapsto\rb{\exp_{\integ,b}\of{\unit_1},\dots,\exp_{\integ,b}\of{\unit_n}}=b\text{.}
\end{equation*}
Trivially,
$u_{\integ,Y}$~is injective.
\item
Let~$\FI,\FJ$~be sequences of finite subsets of~$\nat^n$.
For $f\in\Exp^n_{\integ,Y}\of{F}$
and $d\in\nat$
let
\begin{equation*}
\Toeplitz_d\of{f}
\defeq
\Toeplitz_{\FI,\FJ,d}\of{f}
\defeq
\rb{f\of{\beta-\alpha}}_{\substack{\alpha\in\FI_d\\\beta\in\FJ_d}}
\in
K^{\FI_d\times\FJ_d}\text{.}
\end{equation*}
Since for $f\in\Exp^n_{\integ,Y}\of{F}$
we clearly have $\restr{f}{\nat^n}\in\Exp^n_Y\of{F}$,
we also set
\begin{equation*}
\Hankel_d\of{f}
\defeq
\Hankel_{\FI,\FJ,d}\of{f}
\defeq
\Hankel_{\FI,\FJ,d}\of{\restr{f}{\nat^n}}\text{.}
\end{equation*}
\end{tlist}
\end{definition}

\begin{lemma}%
\label{lemma:hankel-toeplitz-decomposition}
Let~$\FI,\FJ$~be sequences of finite subsets of~$\nat^n$.
Then the following holds:
\begin{tlist}
\item%
\label{lemma:hankel-toeplitz-decomposition:hankel}
Let $Y\subseteq K^n$ and $u=u_Y$.
For $f\in\Exp^n_Y\of{F}$,
$f=\sum_{b\in M}f_bb$
with $M\subseteq B_Y$ finite
and $f_b\in F$,
we have
\begin{equation*}
\Hankel_d\of{f}
=
\transpose{\sumrb{\Vandermonde_{\FI_d}^{u\of{M}}}}\mul C_f\mul\Vandermonde_{\FJ_d}^{u\of{M}}\text{.}
\end{equation*}
Here~$\Vandermonde_{\FI_d}^{u\of{M}}\in K^{u\of{M}\times\FI_d}$
denotes the matrix
of $\ev{\FI_d}{u\of{M}}$ \wrt~the monomial basis of~$S_{\FI_d}$
and the canonical basis of~$K^{u\of{M}}$.
The matrix
$C_f
\in
F^{u\of{M}\times u\of{M}}$
is the diagonal matrix
with the non-zero
coefficients~$f_b$ of~$f$
on the \qquot{diagonal}.
\item%
\label{lemma:hankel-toeplitz-decomposition:toeplitz}
Let $Y\subseteq\rb{K\setminus\set{0}}^n$ and $u=u_{\integ,Y}$.
For $f\in\Exp^n_{\integ,Y}\of{F}$,
$f=\sum_{b\in M}f_bb$
with $M\subseteq B_{\integ,Y}$ finite
and $f_b\in F$,
we have
\begin{equation*}
\Toeplitz_d\of{f}
=
\transpose{\sumrb{\Vandermonde_{\FI_d}^{1/u\of{M}}}}\mul C_f\mul\Vandermonde_{\FJ_d}^{u\of{M}}\text{.}
\end{equation*}
Here~$\Vandermonde_{\FI_d}^{1/u\of{M}}\in K^{1/u\of{M}\times\FI_d}$
denotes the matrix
of
$\ev{\FI_d}{\pset{1/b}{b\in u\of{M}}}$
and~$C_f$~and~$\Vandermonde_{\FJ_d}^{u\of{M}}$ are as in
part~\ref{lemma:hankel-toeplitz-decomposition:hankel}.
\end{tlist}
\end{lemma}

\begin{proof}
This follows by
straightforward computations;
see,~\eg,~%
\cite[Lemma~\textup{2.7}\,\textup{(a)}]{von17}
for
part~\ref{lemma:hankel-toeplitz-decomposition:hankel}
and~%
\cite[Lemma~\textup{2.32}\,\textup{(a)}]{von17}
for
part~\ref{lemma:hankel-toeplitz-decomposition:toeplitz},
respectively.
\end{proof}

The following theorem is a multivariate variant of Prony's method
\trb{\confer~%
Example~\ref{example:classic-prony}}.

\begin{theorem}%
[Prony structures for exponential sums]%
\label{theorem:exponential-sums}
Let~$K$~be a field.
Let~$\FJ$~be a sequence of finite subsets of~$\nat^n$
such that
$\FJ_d\subseteq\FJ_{d+1}$ for all large~$d$ and $\bigcup_{d\in\nat}\FJ_d=\nat^n$.
Let the sequence~$\FI$~be defined by
$\FI_d\defeq\FJ_{\ell\of{d}}$ for an unbounded monotonous sequence $\ell\colon\nat\to\nat$.
Then the following hold,
with $Y\subseteq K^n$ in~\ref{theorem:exponential-sums:hankel-NN^n}
and
$Y\subseteq\rb{K\setminus\set{0}}^n$
in~\ref{theorem:exponential-sums:toeplitz-ZZ^n}
and~\ref{theorem:exponential-sums:hankel-ZZ^n}:
\begin{tlist}
\item%
\label{theorem:exponential-sums:hankel-NN^n}
The map $f\mapsto\rb{\Hankel_d\of{f}}_{d\in\nat}$
induces a Prony structure on~$\Exp^n_Y\of{F}$.
\item%
\label{theorem:exponential-sums:toeplitz-ZZ^n}
The map $f\mapsto\rb{\Toeplitz_d\of{f}}_{d\in\nat}$
induces a Prony structure on~$\Exp^n_{\integ,Y}\of{F}$.
\item%
\label{theorem:exponential-sums:hankel-ZZ^n}
The map $f\mapsto\rb{\Hankel_d\of{f}}_{d\in\nat}$
induces a Prony structure on~$\Exp^n_{\integ,Y}\of{F}$.
\end{tlist}
\end{theorem}

\begin{proof}
In every case we write $u=u_Y$ and $u=u_{\integ,Y}$,
respectively.

\ref{theorem:exponential-sums:hankel-NN^n}
Let
$f\in\Exp^n_Y\of{F}$,
$M\subseteq B_Y$ finite,
and $\rb{f_b}_{b\in M}\in\rb{F\setminus\set{0}}^M$
such that
$f=\sum_{b\in M}f_bb$.
We will verify that
condition~\ref{theorem:prony-structure-characterization:monomorphism}
of
Theorem~\ref{theorem:prony-structure-characterization}
holds
for~$f$~and~$M$~as described in
Remark~\ref{remark:prony-structure-for-generating-systems}.
In particular,
it then follows
that~$B_Y$~is an $F$-basis of~$\Exp^n_Y\of{F}$.

By the assumptions on~$\FJ$~and~$\FI$
and
Lemma~\ref{lemma:polynomial-interpolation},
$\ev{\FI_d}{M}$
is surjective for all large~$d$
and thus
$\transpose{\rb{\ev{\FI_d}{M}}}$
is injective.

Hence,
by
Lemma~\ref{lemma:hankel-toeplitz-decomposition}\,\ref{lemma:hankel-toeplitz-decomposition:hankel},
for all large~$d$
we have the following commutative diagram.
\begin{equation*}
\begin{tikzcd}
|[alias=KJd]|  K^{\FJ_d}  &&                      &&                      &&  |[alias=KId]|  K^{\FI_d}  \\
|[alias=SJd]|  S_{\FJ_d}  &&  |[alias=KM1]|  K^M  &&  |[alias=KM2]|  K^M  &&  |[alias=SId]|  S_{\FI_d}
\ar[  "\Hankel_d\of{f}"                 ,  from=KJd  ,  to=KId  ,        ]
\ar[  "\isom"                           ,  from=KJd  ,  to=SJd  ,        ]
\ar[  "\ev{\FJ_d}{M}"                   ,  from=SJd  ,  to=KM1  ,        ]
\ar[  "\text{$C_f$, $\isom$}"           ,  from=KM1  ,  to=KM2  ,        ]
\ar[  "\transpose{\rb{\ev{\FI_d}{M}}}"  ,  from=KM2  ,  to=SId  ,  hook  ]
\ar[  "\isom"'                          ,  from=SId  ,  to=KId  ,        ]
\end{tikzcd}
\end{equation*}
Thus,
the assertion follows immediately from
Theorem~\ref{theorem:prony-structure-characterization}
together with
Remark~\ref{remark:prony-structure-for-generating-systems}.

\ref{theorem:exponential-sums:toeplitz-ZZ^n}
This follows analogously to
part~\ref{theorem:exponential-sums:hankel-NN^n}
using the elementary fact that
\begin{equation*}
\linrank\sumoftext{\Vandermonde_{\FI_d}^{1/M}}
=
\linrank\sumoftext{\Vandermonde_{\FI_d}^M}
\end{equation*}
\trb{\confer~%
\cite[Lemma~\textup{2.31}]{von17}}
and
with
Lemma~\ref{lemma:hankel-toeplitz-decomposition}\,\ref{lemma:hankel-toeplitz-decomposition:hankel}
replaced by
Lemma~\ref{lemma:hankel-toeplitz-decomposition}\,\ref{lemma:hankel-toeplitz-decomposition:toeplitz}.

\ref{theorem:exponential-sums:hankel-ZZ^n}
This follows immediately from
part~\ref{theorem:exponential-sums:hankel-NN^n}.
\end{proof}

In particular,
for $Y\subseteq K^n$
and
$f\in\Exp^n_Y\of{F}$
the notation
$\suppvar_{u_Y}\of{f}$
is justified
by
Theorem~\ref{theorem:exponential-sums}
\trb{and analogously for $Y\subseteq\rb{K\setminus\set{0}}^n$
and $f\in\Exp^n_{\integ,Y}\of{F}$}.

\begin{remark}
As mentioned above,
one advantage of the Hankel Prony structure~$\Hankel$
over the Toeplitz Prony structure~$\Toeplitz$
is that~$\Hankel$ works with exponential sums with arbitrary bases in~$K^n$
while~$\Toeplitz$ needs bases in~$\rb{K\setminus\set{0}}^n$.

On the other hand,
some relevant results in this context are known
only for Toeplitz matrices;
see,~\eg,~%
\cite[Theorem~\textup{3.7}]{KPRv16}.

In the spirit of
D\'iaz-Kaltofen~%
\cite{DK98}
and
Garg-Schost~%
\cite{GS09},
we discuss one additional advantage of the Toeplitz variant
regarding the number of used evaluations.
Let~$K$~be a field extension of~$F$.
Let~$I$~be a set,
$V\le K^I$ be an $F$-vector space of functions~$I\to K$
and~$B$~be a basis of~$V$.
Moreover,
let
$\varphi\colon K\to K$ be an $F$-automorphism of~$K$
such that
for
$b\in B$
we have
$\varphi\comp b\in B$.
Further,
assume that
a subset $I_0\subseteq I$
is given
together with
a function $\psi\colon I\to I$
such that
$\psi\of{I_0}\subseteq I_1\defeq I\setminus I_0$
and for every $f\in V$
the following diagram is commutative:
\begin{equation*}
\begin{tikzcd}
|[alias=I1]|  I  &&  |[alias=K1]|  K  \\
|[alias=I2]|  I  &&  |[alias=K2]|  K
\ar[  "f"         ,  from=I1  ,  to=K1  ]
\ar[  "\psi"      ,  from=I1  ,  to=I2  ]
\ar[  "\varphi"'  ,  from=K2  ,  to=K1  ]
\ar[  "f"         ,  from=I2  ,  to=K2  ]
\end{tikzcd}
\end{equation*}
\trb{It is of course sufficient to check this
diagram for every $f=b\in B$.}
Thus,
under these assumptions,
one can replace the evaluations of~$f$ at~$\alpha\in I_0$
by evaluations of~$\varphi$~at $f\of{\psi\of{\alpha}}$.
One does not need to evaluate at any element of~$I_0$.

An application is the case
$F=\real$,
$K=\complex$,
and the space
$V=\Exp^n_{\integ,\torus^n}\of{\real}$
of exponential sums
with real coefficients
supported on the
analytic torus
\begin{equation*}
\torus^n
=
\pset{z\in\complex^n}{\text{$\abs{z_j}=1$ for $j=1,\dots,n$}}
\subseteq
\complex^n\text{.}
\end{equation*}
Take $\varphi\colon\complex\to\complex$ to be the complex conjugation
and let
$I=\integ^n$,
$\psi\colon I\to I$,
$\alpha\mapsto-\alpha$,
with
$I_0=\pset{\alpha\in I}{\alpha_1<0}$.
In this case,
one can often define the Toeplitz matrix~$\Toeplitz_{\FI,\FJ,d}\of{f}$
using fewer evaluations than in the Hankel matrix~$\Hankel_{\FI,\FJ,d}\of{f}$.

Let
$f\in\Exp^n_{\integ,\rb{K\setminus\set{0}}^n}\of{F}$
be arbitrary.
Then
the number~$s_{\Hankel,\FI,\FJ,d}$
of evaluations needed to define the Hankel matrix~$\Hankel_{\FI,\FJ,d}\of{f}$
can be different from the number~$s_{\Toeplitz,\FI,\FJ,d}$
of evaluations needed to define~$\Toeplitz_{\FI,\FJ,d}\of{f}$,
depending on the choice of~$\FI$~and~$\FJ$.
In general one has
\begin{equation*}
\text{$s_{\Hankel,\FI,\FJ,d}=\card{\FI_d+\FJ_d}$
\quad
and
\quad
$s_{\Toeplitz,\FI,\FJ,d}=\card{\FJ_d-\FI_d}$.}
\end{equation*}
Thus
for example,
in the bivariate case $n=2$
one has
\begin{equation*}
\text{$s_{\Hankel,\FM,\FM,d}
=
s_{\Toeplitz,\FM,\FM,d}$ for all~$d$}
\end{equation*}
and
\begin{equation*}
\text{$s_{\Hankel,\FT,\FT,2}
=
15
\ne
19
=
s_{\Toeplitz,\FT,\FT,2}$.}
\end{equation*}
A more detailed discussion of this fact can be found in
Josz-Lasserre-Mourrain~%
\cite[Section~\textup{2.3.2}]{JLM19}.

It would be interesting to compare
Prony indices $\Pronyindex_{\Hankel_{\FI,\FJ}}\of{f}$
and $\Pronyindex_{\Toeplitz_{\FI,\FJ}}\of{f}$
of $f\in\Exp^n_{\integ,Y}\of{F}$
for various choices of the involved parameters.
\end{remark}

\section{Applications of Prony structures}%
\label{section:applications}

In this section we discuss
several
reconstruction techniques
in the context of Prony structures,
namely
the Dress-Grabmeier framework~%
\cite{DG91},
the
Grigoriev-Karpinski-Singer~%
\cite{GKS91}
and the related
Peter-Plonka framework~%
\cite{PP13}
\trb{see also
Remark~\ref{remark:frameworks}},
sparse polynomial interpolation~\wrt~the monomial
\trb{Ben-Or/Tiwari~%
\cite{BT88}}
and
Chebyshev bases~%
\cite{LS95,PT14,IKY18,HS20+}
and a sparse technique for
Gau\ss{}ian sums~%
\cite{PPS15}.

The following theorem
casts the
Dress-Grabmeier framework~%
\cite{DG91}
for sparse interpolation of character sums
in terms of Prony structures.

\begin{theorem}%
[Prony structure for character sums]%
\label{theorem:character-sums}
Let
$\rb{M,{+}}$
be a commutative monoid generated by
elements~$a_1,\dots,a_n\in M$.
Consider a set~$B$~of monoid homomorphisms
\trb{\idest,~\emph{characters}}
from~$M$~to~$\rb{K,{\mul}}$,
and let
$V\defeq\mgen{B}$
be the $K$-subvector space of~$K^M$~generated by~$B$.
Let
\begin{equation*}
u\colon B\to K^n\text{,}\quad\chi\mapsto\rb{\chi\of{a_1},\dots,\chi\of{a_n}}\text{.}
\end{equation*}
Let~$\FI,\FJ$~be sequences of finite subsets of~$\nat^n$
with $\FI_d\subseteq\FI_{d+1}$
and $\FJ_d\subseteq\FJ_{d+1}$ for all large~$d$
and $\bigcup_{d\in\nat}\FI_d=\bigcup_{d\in\nat}\FJ_d=\nat^n$.
For $f\in V$ and $d\in\nat$ set
\begin{equation*}
P_d\of{f}
\defeq
\sumrb{f\sumof{\sum_{j=1}^n\rb{\alpha_j+\beta_j}a_j}}_{\substack{\alpha\in\FI_d\\\beta\in\FJ_d}}
\in
K^{\FI_d\times\FJ_d}\text{.}
\end{equation*}
Then~$P_d\of{f}$~induces a Prony structure on~$V$.
\end{theorem}

\begin{proof}
If $u\of{\chi_1}=u\of{\chi_2}$ for characters~$\chi_i$
then $\chi_1\of{a_j}=\chi_2\of{a_j}$ for all $j=1,\dots,n$.
Since~$M$~is generated by $\set{a_1,\dots,a_n}$
this implies $\chi_1=\chi_2$,
and thus~$u$~is injective.
For $f\in V$
write
$f=\sum_{x\in\suppvar_u\of{f}}f_x\chi_x$
with $f_x\in K$
and $\chi_x\in B$
with
$u\of{\chi_x}=x$.
Let
$C
\defeq
\rb{f_x\unit_x}_{x\in\suppvar_u\of{f}}
\in
K^{\suppvar_u\of{f}\times\suppvar_u\of{f}}$.
A computation on the corresponding matrices
shows that one has the following commutative diagram:
\begin{equation*}
\begin{tikzcd}
|[alias=KJd]|  K^{\FJ_d}  &&                                          &&                                          &&  |[alias=KId]|  K^{\FI_d}  \\
|[alias=SJd]|  S_{\FJ_d}  &&  |[alias=Ksupp1]|  K^{\suppvar_u\of{f}}  &&  |[alias=Ksupp2]|  K^{\suppvar_u\of{f}}  &&  |[alias=SId]|  S_{\FI_d}
\ar[  "P_d\of{f}"                                      ,  from=KJd     ,  to=KId     ]
\ar[  "\isom"                                          ,  from=KJd     ,  to=SJd     ]
\ar[  "\ev{\FJ_d}{\suppvar_u\of{f}}"                   ,  from=SJd     ,  to=Ksupp1  ]
\ar[  "\text{$C$, $\isom$}"                            ,  from=Ksupp1  ,  to=Ksupp2  ]
\ar[  "\transpose{\rb{\ev{\FI_d}{\suppvar_u\of{f}}}}"  ,  from=Ksupp2  ,  to=SId     ]
\ar[  "\isom"'                                         ,  from=SId     ,  to=KId     ]
\end{tikzcd}
\end{equation*}
Clearly,
$C$~is invertible,
and thus~$P$~is a Prony structure on~$V$
by
Lemma~\ref{lemma:polynomial-interpolation},
Theorem~\ref{theorem:prony-structure-characterization},
and
Remark~\ref{remark:prony-structure-for-generating-systems}.
\end{proof}

\begin{remark}
\begin{tlist}
\item
Since
${\exp_b}\in\Hom\of{\rb{\nat^n,{+}},\rb{K,{\mul}}}$,
the Dress-Grabmeier framework
contains the Prony structures for exponential sums.
\item
Note that
Dress-Grabmeier
allows more generally arbitrary monoids
whereas in
Theorem~\ref{theorem:character-sums}
we allow only finitely generated ones.
Roughly speaking,
in applications to function spaces this corresponds
to allowing only a fixed finite number~$n$~of variables.
This is no restriction in any case we have in mind.
\item
Note that
Dress-Grabmeier
implies the Dedekind independence lemma,
\idest,~%
that any set of monoid characters is linearly independent.
\end{tlist}
\end{remark}

Next we present a family of
methods
that was given in the case of one operator
in
Peter-Plonka~%
\cite{PP13}.
See also
Mourrain~%
\cite{Mou18}
for related discussions in the multivariate case
and the book of
Plonka, Potts, Steidl, and Tasche~%
\cite[Section~\textup{10.4.2}]{PPST18}.
Essentially,
it is a generalization of the framework given by
Grigoriev, Karpinski, and Singer~%
\cite{GKS91}
for the case of~$\varDelta$~being a point evaluation functional.
We derive our statement directly from
Theorem~\ref{theorem:character-sums}.

As usual,
the
point spectrum
of an endomorphism
$\varphi\in\End_K\of{W}$
of a $K$-vector space~$W$
is denoted by
\begin{equation*}
\pspec\of{\varphi}
=
\pset{\lambda\in K}{\ker\of{\varphi-\lambda\id_W}\ne\set{0}}
\end{equation*}
and for
$\lambda\in\pspec\of{\varphi}$
let
\begin{equation*}
W^\varphi_\lambda
=
\ker\of{\varphi-\lambda\id_W}
\end{equation*}
be the eigenspace of~$\varphi$~\wrt~$\lambda$.
For pairwise commuting operators
$\varphi_1,\dots,\varphi_n\in\End_K\of{W}$
and
$\alpha\in\nat^n$
we use the notation
\begin{equation*}
\varphi^\alpha\defeq\varphi_1^{\alpha_1}\comp\dots\comp\varphi_n^{\alpha_n}\in\End_K\of{W}\text{.}
\end{equation*}

\begin{corollary}%
[Prony structure for eigenvector sums]%
\label{corollary:eigenvector-sums-multivariate}
Let
$\varphi_1,\dots,\varphi_n\in\End_K\of{W}$
be pairwise commuting operators
and
consider
$\varLambda\subseteq\prod_{j=1}^n\pspec\of{\varphi_j}$.
Assume that
for every
$\lambda\in\varLambda$
we have
$\bigcap_{j=1}^nW^{\varphi_j}_{\lambda_j}\ne\set{0}$
and
choose
\begin{equation*}
b_\lambda\in\bigcap_{j=1}^nW^{\varphi_j}_{\lambda_j}\setminus\set{0}\text{.}
\end{equation*}
Let
\begin{equation*}
\text{$B\defeq\pset{b_\lambda}{\lambda\in\varLambda}$,
\quad
$V\defeq\mgen{B}_K$,
\quad
and
\quad
$u\colon B\to K^n$,
\quad
$b_\lambda\mapsto\lambda$.}
\end{equation*}
Let
$\varDelta\in W^\ast=\Hom_K\of{W,K}$
be such that
\begin{equation*}
V\cap\ker\of{\varDelta}
=
\set{0}\text{.}
\end{equation*}
Let~$\FI,\FJ$~be sequences of finite subsets of~$\nat^n$
with $\FI_d\subseteq\FI_{d+1}$ and $\FJ_d\subseteq\FJ_{d+1}$ for all large~$d$
and $\bigcup_{d\in\nat}\FI_d=\bigcup_{d\in\nat}\FJ_d=\nat^n$.
For $f\in V$ and $d\in\nat$
set
\begin{equation*}
P_d\of{f}
\defeq
\rb{\varDelta\of{\varphi^{\alpha+\beta}\of{f}}}_{\substack{\alpha\in\FI_d\\\beta\in\FJ_d}}\in K^{\FI_d\times\FJ_d}\text{.}
\end{equation*}
Then~$P_d\of{f}$~induces a Prony structure on~$V$.
\end{corollary}

\begin{proof}
We apply
Theorem~\ref{theorem:character-sums}
similarly as in
Grigoriev, Karpinski, and Singer~%
\cite[p.~\textup{78f}]{GKS91}.
Let~$M$~denote the submonoid of~$\rb{\End_K\of{W},{\comp}}$
generated by $\varphi_1,\dots,\varphi_n$.
For
$\lambda\in\varLambda$
let
\begin{equation*}
\chi_\lambda\colon M\to K\text{,}
\quad
\varphi^\alpha\mapsto\frac{\varDelta\of{\varphi^\alpha\of{b_\lambda}}}{\varDelta\of{b_\lambda}}\text{.}
\end{equation*}
Clearly,
$\chi_\lambda$~is well-defined.
Since $\chi_\lambda\of{\varphi^\alpha}=\lambda^\alpha$ for every $\alpha\in\nat^n$,
$\chi_\lambda$ is a monoid homomorphism $M\to\rb{K,{\mul}}$.
Thus,
by
Theorem~\ref{theorem:character-sums},
$Q_d\of{f}=\rb{f\of{\varphi^{\alpha+\beta}}}_{\alpha\in\FI_d,\beta\in\FJ_d}$
induces a Prony structure on the vector space
$U\defeq\pmgen{\chi_\lambda}{\lambda\in\varLambda}_K\le K^M$
with respect to $v\colon\chi_\lambda\mapsto\lambda$.
Since $\suppvar_u\of{b_\lambda}=\lambda=\suppvar_v\of{\chi_\lambda}$,
the assertion follows.
\end{proof}

Observe that there are interesting situations
where the condition that the~$b_\lambda$'s~can be chosen in the desired way
is fulfilled.
For example this is the case if~$W$~is
a finite-dimensional $\complex$-vector space
see,~\eg,~%
Horn-John\-son~%
\cite[Lemma~\textup{1.3.19}]{HJ13}.

With a little more effort in a direct proof,
one can avoid the commutativity assumption
in
Corollary~\ref{corollary:eigenvector-sums-multivariate}
\trb{but of course one still needs that
$\bigcap_{j=1}^nW^{\varphi_j}_{\lambda_j}\ne\set{0}$
for every
$\lambda\in\varLambda$}.

The case $n=1$ identifies
the method in~%
\cite{PP13}
as a Prony structure.

\begin{corollary}%
[Peter-Plonka~%
{\cite[Theorem~\textup{2.1}]{PP13}}]%
\label{corollary:eigenvector-sums-univariate}
Let $\varphi\in\End_K\of{W}$
and
consider
$\varLambda
\subseteq
\pspec\of{\varphi}$.
For $\lambda\in\varLambda$
choose
\begin{equation*}
b_\lambda\in W^\varphi_\lambda\setminus\set{0}\text{.}
\end{equation*}
Let
\begin{equation*}
\text{$B\defeq\pset{b_\lambda}{\lambda\in\varLambda}$,
\quad
$V\defeq\mgen{B}_K$,
\quad
and
\quad
$u\colon B\to K$,
\quad
$b_\lambda\mapsto\lambda$.}
\end{equation*}
Let
$\varDelta\in W^\ast$
be such that
\begin{equation*}
V\cap\ker\of{\varDelta}=\set{0}\text{.}
\end{equation*}
For $f\in V$
and $d\in\nat$
set
\begin{equation*}
P_d\of{f}
\defeq
\rb{\varDelta\of{\varphi^{\alpha+\beta}\of{f}}}_{\substack{\alpha=0,\dots,d-1\\\beta=0,\dots,d}}
\in
K^{d\times\rb{d+1}}\text{.}
\end{equation*}
Then~$P_d\of{f}$~induces a Prony structure on~$V$.
\end{corollary}

\begin{proof}
Take
$n=1$,
$\FI_d=\FT_{d-1}$
and
$\FJ_d=\FT_d$
in
Corollary~\ref{corollary:eigenvector-sums-multivariate}.
\end{proof}

\begin{example}
Several applications
for various choices of the endomorphism~$\varphi$~and
the functional~$\varDelta$~can be found in~%
\cite{PP13},
for example,
with~$\varphi\in\End\of{W}$~chosen
as a
Sturm-Liouville differential operator
\trb{$W=\cont^\infty\of{\real}$}
or as a diagonal matrix with distinct elements on the diagonal
\trb{$W=K^n$}.
\end{example}

\begin{remark}%
\label{remark:multiplicities}
Besides
Corollary~\ref{corollary:eigenvector-sums-univariate},
Peter-Plonka~%
\cite[Theorem~\textup{2.4}]{PP13}
extended their method,~\eg,~to include generalized eigenvectors
and multiplicities;
see also
Mourrain~%
\cite{Mou18}
and
Stampfer-Plonka~%
\cite{SP20+}.
At present Prony structures do not cover this variation.
Since all examples we have in mind and which are discussed in this manuscript
do not use generalized eigenvectors and multiplicities,
we omit a detailed discussion here.
See also Remark~\ref{remark:radical-ideals}.
\end{remark}

The following lemma
singles out a simple
transfer principle for Prony structures
that will be applied in
Corollary~\ref{corollary:monomial-sparse-interpolation}
and
Corollary~\ref{corollary:chebyshev-sparse-interpolation}.
It is also one motivation
for the introduction of Prony maps
in
Section~\ref{section:prony-maps}.

\begin{lemma}%
[Transfer principle for Prony structures]%
\label{lemma:transfer-principle}
Let
$V,\widetilde{V}$
be $F$-vector spaces
with bases~$B,\widetilde{B}$,
respectively,
and let
$u\colon B\to K^n$
and
$\widetilde{u}\colon\widetilde{B}\to K^n$
be injective.
Let
$\varphi\colon V\to\widetilde{V}$
\trb{not necessarily linear}
and for every $f\in V$ let
\begin{equation*}
\suppvar_u\of{f}
=
\suppvar_{\widetilde{u}}\of{\varphi\of{f}}\text{.}
\end{equation*}
Then every Prony structure~$\widetilde{P}$~on~$\widetilde{V}$
induces a Prony structure~$\varphi^\ast\of{\widetilde{P}}$~on~$V$
with
\begin{equation*}
\varphi^\ast\of{\widetilde{P}}_d\of{f}
\defeq
\widetilde{P}_d\of{\varphi\of{f}}
\end{equation*}
for $f\in V$ and $d\in\nat$.
The following commutative diagram illustrates the situation.
\begin{equation*}
\begin{tikzcd}
|[alias=V1]|  V  &&  |[alias=V2     ]|  \widetilde{V}                                                 \\
                 &&  |[alias=product]|  \prod_{d\in\nat}K^{\widetilde{\FI}_d\times\widetilde{\FJ}_d}
\ar[  "\varphi"                          ,  from=V1  ,  to=V2       ]
\ar[  "\widetilde{P}"                    ,  from=V2  ,  to=product  ]
\ar[  "\varphi^\ast\of{\widetilde{P}}"'  ,  from=V1  ,  to=product  ]
\end{tikzcd}
\end{equation*}
\end{lemma}

\begin{proof}
Let
$P\defeq\varphi^\ast\of{\widetilde{P}}$.
By the hypotheses,
for $f\in V$
and
all large~$d$
we have
\begin{equation*}
\suppvar_u\of{f}
=
\suppvar_{\widetilde{u}}\of{\varphi\of{f}}
=
\ZL\of{\ker\widetilde{P}_d\of{\varphi\of{f}}}
=
\ZL\of{\ker P_d\of{f}}
\end{equation*}
and
\begin{equation*}
\I_{\widetilde{\FJ}_d}\of{\suppvar_u\of{f}}
=
\I_{\widetilde{\FJ}_d}\of{\suppvar_{\widetilde{u}}\of{\varphi\of{f}}}
\subseteq
\ker\of{\widetilde{P}_d\of{\varphi\of{f}}}
=
\ker\of{P_d\of{f}}\text{.}
\end{equation*}
This concludes the proof.
\end{proof}

The following corollary identifies
a well-known sparse interpolation technique for polynomials~\wrt~the monomial basis
\trb{see,~\eg,~%
\cite[Section~\textup{5.4}]{Mou18}}
as a Prony structure.
In particular,
the framework of Prony structures allows a simultaneous proof
of the Hankel and Toeplitz cases.
There are analogous results for the Chebyshev basis
\trb{see
Corollary~\ref{corollary:chebyshev-sparse-interpolation}}.

Let~$F$~be a field
and
consider
\begin{equation*}
V\defeq F\ringad{\Y_1,\dots,\Y_n}
\end{equation*}
as an $F$-vector space
with the monomial basis
\begin{equation*}
B\defeq\pset{\Y^\alpha}{\alpha\in\nat^n}\text{.}
\end{equation*}
Choose a field extension~$K$~of~$F$~and let
$b\in\rb{K\setminus\set{0}}^n$
be such that the function
\begin{equation*}
\text{$u\colon B\to K^n$,
\quad
$\Y^\alpha\mapsto\rb{b_1^{\alpha_1},\dots,b_n^{\alpha_n}}$,}
\end{equation*}
is injective.%
\footnote{%
For example,
for $K=F=\complex$,
any $b\in\complex^n$ such that
$b_j\ne0$ and $b_j$~is not a root of unity for all $j=1,\dots,n$ works.
Of course,
$K$~cannot be finite,
for otherwise
$u\colon B\to K^n$
cannot be injective.
One may always choose $K\defeq F\fieldad{\W}$
\trb{with~$\W$~an indeterminate over~$F$}
and
$b\defeq\rb{\W,\dots,\W}\in K^n$.%
}
Observe that then necessarily
$u\of{B}\subseteq\rb{K\setminus\set{0}}^n$.

Moreover,
set
$\widetilde{V}\defeq\Exp^n_{\integ,u\of{B}}\of{F}$,
$\widetilde{B}\defeq\pset{\exp_b}{b\in u\of{B}}$,
and
$\widetilde{u}\colon\widetilde{B}\to K^n$, ${\exp_b}\mapsto b$.

\begin{corollary}%
[Prony structures for sparse polynomial interpolation]%
\label{corollary:monomial-sparse-interpolation}
For $p\in V$
let
\begin{equation*}
f_p\colon\integ^n\to K\text{,}
\quad
\alpha\mapsto p\of{b_1^{\alpha_1},\dots,b_n^{\alpha_n}}
\text{~\trb{$=p\of{u\of{\Y^\alpha}}$ if $\alpha\in\nat^n$}.}
\end{equation*}
Then the following holds:
\begin{tlist}
\item%
\label{corollary:monomial-sparse-interpolation:phi(f)-exponential-sum}
For all $p\in V$
we have
$f_p\in\widetilde{V}$
and
$\varphi\colon V\to\widetilde{V}$,
$p\mapsto f_p$,
is $F$-linear.
\item%
\label{corollary:monomial-sparse-interpolation:supp(phi(f))}
For all $p\in V$
we have
$\suppvar_{\widetilde{u}}\of{f_p}=\suppvar_u\of{p}$.
\end{tlist}
Hence,
any Prony structure on~$\widetilde{V}$
\trb{in particular
the Prony structures
from
Theorem~\ref{theorem:exponential-sums}},
induces a Prony structure on~$V$
by the transfer principle
\trb{Lemma~\ref{lemma:transfer-principle}}.
\end{corollary}

\begin{proof}
\ref{corollary:monomial-sparse-interpolation:phi(f)-exponential-sum}
Let
$\polysupp\of{p}=\pset{\beta\in\nat^n}{\Y^\beta\in\basissupp_B\of{p}}$.
For
$\alpha\in\integ^n$
and using
Definition~\ref{definition:exponential-sums-with-domain-NN^n}
we have
\begin{equation*}
f_p\of{\alpha}
=
\sum_{\beta\in\polysupp\of{p}}p_\beta\mul\rb{b_1^{\alpha_1},\dots,b_n^{\alpha_n}}^\beta
=
\sum_{\beta\in\polysupp\of{p}}p_\beta\mul\rb{b_1^{\beta_1},\dots,b_n^{\beta_n}}^\alpha
=
\sum_{\beta\in\polysupp\of{p}}p_\beta\mul\exp_{\integ,u\of{\Y^\beta}}\of{\alpha}\text{.}
\end{equation*}
This shows that
$f_p\in\Exp^n_{\integ,u\of{B}}\of{F}=\widetilde{V}$.
In particular,
$\varphi$~is well-defined.
The linearity of~$\varphi$~follows immediately from the definition.

\ref{corollary:monomial-sparse-interpolation:supp(phi(f))}
Since~$u$~is injective,
the computation
in the proof of
part~\ref{corollary:monomial-sparse-interpolation:phi(f)-exponential-sum}
shows that
\begin{equation*}
\suppvar_{\widetilde{u}}\of{f_p}
=
\pset{u\of{\Y^\beta}}{\beta\in\polysupp\of{p}}
=
\pset{u\of{m}}{m\in\basissupp_B\of{p}}
=
\suppvar_u\of{p}\text{.}
\end{equation*}
This concludes the proof.
\end{proof}

\begin{example}
The reconstruction method for $p\in V=F\ringad{\Y_1,\dots,\Y_n}$
from
Corollary~\ref{corollary:monomial-sparse-interpolation}
is efficient if~$p$~has small rank,
\idest,~is a \qquot{sparse polynomial}.
To give an illustration,
let
$n=2$,
$b\in\rb{K\setminus\set{0}}^n$
be chosen appropriately
and
$p=\Y^\beta-\Y^\gamma\in V$
be a binomial.
Then $\basisrank\of{f_p}=2$,
hence
the polynomial~$p$~can be reconstructed,
independently of its degree,
from the
$\card{\FT_3}
=
\binom{n+3}{3}
=
\binom{5}{3}
=
10$
evaluations used for the matrix
$\Hankel_{\FT_1,\FT_2}\of{f_p}$.

The number of evaluations of~$p$
can be further reduced if~$p$~is known to be of degree at most~$d-1$.
In this case,
$q\defeq p\of{\Z,\Z^d,\dots,\Z^{d^{n-1}}}\in F\ringad{\Z}$
is a binomial of degree at most~$d^n-1$
in~\emph{one}~variable.
The above binomial can thus be reconstructed from four evaluations.
\end{example}

Let
$\cheb_i\in\integ\ringad{\Y}$
denote the~$i$-th~Chebyshev polynomial
\trb{\idest,~%
$\cheb_0=1$,
$\cheb_1=\Y$,
and
$\cheb_i=2\Y\cheb_{i-1}-\cheb_{i-2}$
for $i\ge2$}.
It is well-known
\trb{and immediate}
that
$B\defeq\pset{\cheb_i}{i\in\nat}$
is a $\rat$-basis of
$V\defeq\rat\ringad{\Y}$.

Decomposing a polynomial $f\in\rat\ringad{\Y}$~%
\wrt~the Chebyshev basis~$B$~is in principle possible
by first decomposing~$f$~in terms of the monomial basis
\trb{Corollary~\ref{corollary:monomial-sparse-interpolation}}
and then computing the Chebyshev decomposition from that.
However,
the natural assumption of an upper bound on the rank of~$f$~\wrt~$B$
does not imply an upper bound on the rank of~$f$~\wrt~the monomial basis,
so that it may be impossible to check the premises of
Corollary~\ref{corollary:monomial-sparse-interpolation}.
Even if such a bound would be given,
efficiency would be a concern.
Lakshman and Saunders~%
\cite{LS95}
proposed a sparse method to compute Chebyshev decompositions directly,
which we recast in the framework of Prony structure in the following.
We first prove a Prony structure for an analogue of exponential sums
in the Chebyshev setting
\trb{Theorem~\ref{theorem:chebyshev-exponential-sums}}.
The Prony structure for Chebyshev-sparse polynomial interpolation
of Lakshman and Saunders~%
\cite{LS95}
then follows in exactly the same way as for
\qquot{monomial-sparse}
polynomial interpolation
\trb{Corollary~\ref{corollary:chebyshev-sparse-interpolation}}.

As observed in
Lakshman-Saunders~%
\cite[p.~\textup{390}]{LS95},
the crucial properties of the Chebyshev polynomials
for their Prony structures
are that for all $i,j\in\nat$ one has the
\emph{linearization relation}
\begin{equation}%
\label{equation:chebyshev-linearization}
\cheb_i\mul\cheb_j
=
\frac{1}{2}\rb{\cheb_{i+j}+\cheb_{\abs{i-j}}}
\end{equation}
and the
\emph{commutativity relation}
\begin{equation}%
\label{equation:chebyshev-commutativity}
\cheb_i\of{\cheb_j}
=
\cheb_j\of{\cheb_i}\text{.}
\end{equation}

The following definition is the Chebyshev analogue
of the exponentials of
Section~\ref{section:exponential-sums}.

\begin{definition}
Let~$F$~be a field of characteristic zero
and~$K$~be a field extension of~$F$.
For $b\in K$
call the function
\begin{equation*}
{\chebexp_b}\colon\nat\to K\text{,}\quad i\mapsto\cheb_i\of{b}\text{,}
\end{equation*}
\emph{Chebyshev exponential with base~$b$}
and for a subset $Y\subseteq K$
denote by
\begin{equation*}
\chebExp_Y\of{F}\defeq\pmgen{{\chebexp_b}}{b\in Y}_F
\end{equation*}
the $F$-vector space of
\emph{Chebyshev exponential sums with bases in~$Y$}.
\end{definition}

\begin{remark}
Observe that considered merely as vector spaces,
$\Exp_Y\of{F}$
and
$\chebExp_Y\of{F}$
are identical.
However,
here we consider them equipped with the bases
of exponentials and Chebyshev exponentials,
respectively,
and provide the notation to keep track of this difference.
\end{remark}

\begin{theorem}%
[Prony structures for Chebyshev exponential sums]%
\label{theorem:chebyshev-exponential-sums}
For
$f\in\chebExp_Y\of{F}$
and
$d\in\nat$
let
\begin{equation*}
P^\prime_d\of{f}
\defeq
\rb{f\of{i+j}+f\of{\abs{i-j}}}_{\substack{i=0,\dots,d-1\\j=0,\dots,d}}
\in
K^{d\times\rb{d+1}}
\end{equation*}
\trb{which is the sum of a Hankel and a Toeplitz matrix}.
Let~$\psi\in\rat^{\rb{d+1}\times\rb{d+1}}$~be the change of basis from the monomial to the Chebyshev basis
and
\begin{equation*}
P_d\of{f}\defeq P^\prime_d\of{f}\mul\psi\text{.}
\end{equation*}
Then~$P_d\of{f}$~induces a Prony structure on~$\chebExp_Y\of{F}$~\wrt~%
\begin{equation*}
u\colon B\to K\text{,}\quad\chebexp_b\mapsto\chebexp_b\of{1}=b\text{.}
\end{equation*}
\end{theorem}

\begin{proof}
The injectivity of~$u$~follows immediately from the definition.

Let $S\defeq K\ringad{\X}$.
The lower part of the following diagram is commutative
by a computation analogous to
Lakshman-Saunders~%
\cite[proof of
Lemma~\textup{6}]{LS95}
\trb{using the linearization relation~\eqref{equation:chebyshev-linearization} above},
where the vertical isomorphisms are those given by the basis
$\set{\cheb_0,\dots,\cheb_d}$
of~$S_{\le d}$
and~$C$~is the isomorphism given by the diagonal matrix
$C\defeq\rb{2f_{T}\unit_T}_{T\in\basissupp_B\of{f}}$.
\begin{equation*}
\begin{tikzcd}
|[alias=Kd+1']|  K^{d+1}    &&                                          &&                                          &&  |[alias=Kd' ]|  K^d          \\
|[alias=Kd+1 ]|  K^{d+1}    &&                                          &&                                          &&  |[alias=Kd  ]|  K^d          \\
|[alias=Sd   ]|  S_{\le d}  &&  |[alias=Ksupp1]|  K^{\suppvar_u\of{f}}  &&  |[alias=Ksupp2]|  K^{\suppvar_u\of{f}}  &&  |[alias=Sd-1]|  S_{\le d-1}
\ar[  "P_d\of{f}"                                        ,  from=Kd+1'   ,  to=Kd'     ]
\ar[  equal                                              ,  from=Kd      ,  to=Kd'     ]
\ar[  "\text{$\psi$, $\isom$}"                           ,  from=Kd+1'   ,  to=Kd+1    ]
\ar[  "P^\prime_d\of{f}"                                 ,  from=Kd+1    ,  to=Kd      ]
\ar[  "\isom"                                            ,  from=Kd+1    ,  to=Sd      ]
\ar[  "\isom"'                                           ,  from=Sd-1    ,  to=Kd      ]
\ar[  "\ev{\le d}{\suppvar_u\of{f}}"                     ,  from=Sd      ,  to=Ksupp1  ]
\ar[  "\text{$C$, $\isom$}"                              ,  from=Ksupp1  ,  to=Ksupp2  ]
\ar[  "\transpose{\rb{\ev{\le d-1}{\suppvar_u\of{f}}}}"  ,  from=Ksupp2  ,  to=Sd-1    ]
\end{tikzcd}
\end{equation*}
The upper part of the diagram is commutative
by the definition of~$P_d\of{f}$
and thus
the assertion follows
from
Lemma~\ref{lemma:polynomial-interpolation},
Theorem~\ref{theorem:prony-structure-characterization},
and
Remark~\ref{remark:prony-structure-for-generating-systems}.
\end{proof}

It is now straightforward
to derive a well-known sparse interpolation technique
for polynomials~\wrt~the Chebyshev basis
\trb{see,~\eg,~%
Lakshman-Saunders~%
\cite{LS95}}
by transferring
the Prony structure for Chebyshev exponential sums
from
Theorem~\ref{theorem:chebyshev-exponential-sums}
to the space of polynomials
using
Lemma~\ref{lemma:transfer-principle}.
To this end,
let~$F$~be a field
of characteristic zero
and
consider
\begin{equation*}
V\defeq F\ringad{\Y}
\end{equation*}
as an $F$-vector space
with the Chebyshev basis
\begin{equation*}
B\defeq\pset{\cheb_i}{i\in\nat}\text{.}
\end{equation*}
Choose a field extension~$K$~of~$F$~and let
$b\in K$
be such that the function
\begin{equation*}
\text{$u\colon B\to K$,
\quad
$\cheb_i\mapsto\cheb_i\of{b}$,}
\end{equation*}
is injective.%
\footnote{%
A choice that always works is $b\in\rat\subseteq F$ with $b>1$.%
}

Moreover,
set
$\widetilde{V}\defeq\chebExp_{u\of{B}}\of{F}$,
$\widetilde{B}\defeq\pset{\chebexp_b}{b\in u\of{B}}$,
and
$\widetilde{u}\colon\widetilde{B}\to K$,
${\chebexp_b}\mapsto b$.

\begin{corollary}%
[Prony structure for Chebyshev-sparse polynomial interpolation]%
\label{corollary:chebyshev-sparse-interpolation}
For $p\in V$
let
\begin{equation*}
f_p\colon\nat\to K\text{,}
\quad
i\mapsto
p\of{u\of{\cheb_i}}\text{.}
\end{equation*}
Then the following holds:
\begin{tlist}
\item%
\label{corollary:chebyshev-sparse-interpolation:phi(f)-chebyshev-exponential-sum}
For all $p\in V$
we have
$f_p\in\widetilde{V}$
and
$\varphi\colon V\to\widetilde{V}$,
$p\mapsto f_p$,
is $F$-linear.
\item%
\label{corollary:chebyshev-sparse-interpolation:supp(phi(f))}
For all $p\in V$
we have
$\suppvar_{\widetilde{u}}\of{f_p}=\suppvar_u\of{p}$.
\end{tlist}
Hence,
any Prony structure on~$\widetilde{V}$
\trb{in particular
the Prony structure
from
Theorem~\ref{theorem:chebyshev-exponential-sums}},
induces a Prony structure on~$V$
by the transfer principle
\trb{Lemma~\ref{lemma:transfer-principle}}.
\end{corollary}

\begin{proof}
\ref{corollary:chebyshev-sparse-interpolation:phi(f)-chebyshev-exponential-sum}
Let
$\polysupp\of{p}=\pset{j\in\nat}{\cheb_j\in\basissupp_B\of{p}}$.
Using the
commutativity relation~\eqref{equation:chebyshev-commutativity}
mentioned above,
for $i\in\nat$ we have
\begin{equation*}
f_p\of{i}
=
\sum_{j\in\polysupp\of{p}}p_j\mul\cheb_j\of{\cheb_i\of{b}}
=
\sum_{j\in\polysupp\of{p}}p_j\mul\cheb_i\of{\cheb_j\of{b}}
=
\sum_{j\in\polysupp\of{p}}p_j\mul\chebexp_{u\of{\cheb_j}}\of{i}\text{.}
\end{equation*}
This shows that
$f_p\in\widetilde{V}$.
In particular,
$\varphi$~is well-defined.
The linearity of~$\varphi$~follows immediately from the definition.

\ref{corollary:chebyshev-sparse-interpolation:supp(phi(f))}
Since~$u$~is injective,
the computation
in the proof of
part~\ref{corollary:chebyshev-sparse-interpolation:phi(f)-chebyshev-exponential-sum}
shows that
\begin{equation*}
\suppvar_{\widetilde{u}}\of{f_p}
=
\pset{u\of{\cheb_j}}{j\in\polysupp\of{p}}
=
\pset{u\of{T}}{T\in\basissupp_B\of{p}}
=
\suppvar_u\of{p}\text{.}
\end{equation*}
This concludes the proof.
\end{proof}

\begin{remark}%
\label{remark:chebyshev-variants}
While versions of
Theorem~\ref{theorem:chebyshev-exponential-sums}
hold
for any basis of polynomials satisfying
a linearization relation with fixed coefficients
for products
\trb{see
Corollary~\ref{corollary:spherical-harmonics}
for a variant in the relative setting of
Section~\ref{section:relative-prony-structures}},
it is in general
not easily possible to obtain corresponding versions of
Corollary~\ref{corollary:chebyshev-sparse-interpolation},
\idest~sparse interpolation techniques,
since bases satisfying commutativity relations are rather elusive
and these conditions are not straightforward to replace.
However,
there are variants for other kinds of Chebyshev bases,
see,~\eg~%
Potts-Tasche~%
\cite{PT14}
and
Imamoglu-Kaltofen-Yang~%
\cite{IKY18}.

Peter and Plonka show how to view Chebyshev polynomials of the first kind
as eigenfunctions of a suitable endomorphism of the space~$W$~of
continuous real-valued functions on the interval~$\lcrc{-1}{1}$,
see~%
\cite[Remark~\textup{4.6}]{PP13}.
Thus,
also the \qquot{analytic} reconstruction technique for these functions given in~%
\cite{PT14}
is recast in the framework for eigenfunction sums.
It is however not clear how this might be translated into a purely algebraic version.

Multivariate variants for Chebyshev polynomials of first and second kind
can be found in a very recent preprint of Hubert and Singer~%
\cite{HS20+}.
\end{remark}

\begin{example}
We give a toy example computation
to illustrate
Corollary~\ref{corollary:chebyshev-sparse-interpolation}.
Let
\begin{equation*}
f=\Y^3\in\rat\ringad{\Y}\text{.}
\end{equation*}
\trb{The polynomial~$f=1/8\mul\cheb_3+1/4\mul\cheb_1$~has Chebyshev rank~$2$.}
We choose $b\defeq2$.
Then we have
\begin{equation*}
P^\prime_2\of{f}
=
\begin{pmatrix}
1  &    8  &    343  \\
8  &  343  &  17576
\end{pmatrix}
+
\begin{pmatrix}
1  &  8  &  343  \\
8  &  1  &    8
\end{pmatrix}  \\
=
\begin{pmatrix}
 2   &   16  &    686  \\
16   &  344  &  17584
\end{pmatrix}
\end{equation*}
and
\begin{equation*}
P_2\of{f}
=
P^\prime_2\of{f}
\mul
\begin{pmatrix}
1  &  0  &  \frac{1}{2}  \\
0  &  1  &  0            \\
0  &  0  &  \frac{1}{2}
\end{pmatrix}
=
\begin{pmatrix}
 2  &   16  &   344  \\
16  &  344  &  8800
\end{pmatrix}
\sim
\begin{pmatrix}
1  &  8  &  172  \\
0  &  1  &   28
\end{pmatrix}
\text{.}
\end{equation*}
Thus,
\begin{equation*}
\ker\of{P_2\of{f}}
=
\mgen{\transpose{\rb{52,-28,1}}}
=
\mgen{\X^2-28\X+52}
=
\mgen{\rb{\X-2}\rb{\X-26}}\text{,}
\end{equation*}
and we recover the support of~$f$~as
\begin{align*}
\basissupp_B\of{f}
&=
u^{-1}\of{\suppvar_u\of{f}}
=
u^{-1}\of{\ZL\of{\ker P_2\of{f}}}
=
u^{-1}\of{\set{2,26}}
=
u^{-1}\of{\set{\cheb_1\of{b},\cheb_3\of{b}}}  \\
&=
\set{\cheb_1,\cheb_3}\text{.}
\end{align*}
If desired,
the coefficients~$1/4$~and~$1/8$~can now be easily computed
by solving a $2\times2$-system of linear equations.
\end{example}

\begin{remark}%
\label{remark:framework-inclusions}
Summarizing the preceding discussion on frameworks
for character~%
\cite{DG91}
and
eigenfunction/eigenvector sums~%
\cite{GKS91,PP13}
and the algebraic and analytic sparse polynomial interpolation techniques~\wrt~the Chebyshev basis~%
\cite{LS95,KY07}
and~%
\cite{PT14},
we obtain the following diagram
of \qquot{inclusions}.
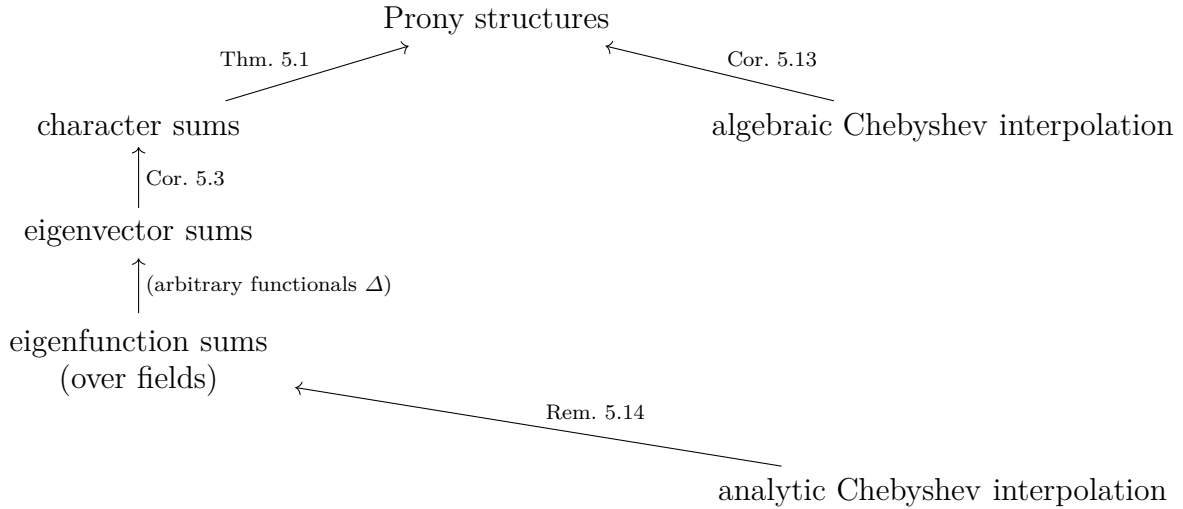
\begin{equation*}
\begin{tikzcd}
                                                                                                          &  |[alias=prony]|  \text{Prony structures}  &                                                                           \\
|[alias=characters    ]|  \text{character sums}                                                           &                                            &  |[alias=chebyshev-algebraic]|  \text{algebraic Chebyshev interpolation}  \\
|[alias=eigenvectors  ]|  \text{eigenvector sums}                                                         &                                            &                                                                           \\
|[alias=eigenfunctions]|  \begin{array}{c}\text{eigenfunction sums}\\\text{\trb{over fields}}\end{array}  &                                            &                                                                           \\
                                                                                                          &                                            &  |[alias=chebyshev-analytic ]|  \text{analytic Chebyshev interpolation}
\ar[  "\text{Thm.~\ref{theorem:character-sums}}"                     ,  from=characters           ,  to=prony           ]
\ar[  "\text{Cor.~\ref{corollary:chebyshev-sparse-interpolation}}"'  ,  from=chebyshev-algebraic  ,  to=prony           ]
\ar[  "\text{Cor.~\ref{corollary:eigenvector-sums-multivariate}}"'   ,  from=eigenvectors         ,  to=characters      ]
\ar[  "\text{\trb{arbitrary functionals~$\varDelta$}}"'              ,  from=eigenfunctions       ,  to=eigenvectors    ]
\ar[  "\text{Rem.~\ref{remark:chebyshev-variants}}"'                 ,  from=chebyshev-analytic   ,  to=eigenfunctions  ]
\end{tikzcd}
\end{equation*}
Lakshman and Saunders
remark on the possibility to \qquot{reconcile}
the frameworks for character or eigenfunction sums
with their algorithm for sparse polynomial interpolation~\wrt~the Chebyshev basis~%
\cite[p.~\textup{388}]{LS95}.
As the framework of Prony structures is of a very general nature,
we would not propose it as a final answer to this question.
However,
it can be hoped that it will be helpful in finding more particular reconciliations.
See also
Remark~\ref{remark:chebyshev-variants}.
\end{remark}

\begin{remark}
For sparse interpolation in various bases
probabilistic results are known
in the literature
under the name
\qquot{early termination},
see for example
Kaltofen-Lee~%
\cite{KL03}.
In the language of the present note,
there the quest is to find probabilistic estimates
of the Prony index~$\Pronyindex_P\of{f}$
of a polynomial~$f$
where the Prony structure~$P$
is given in similar ways as in
Corollary~\ref{corollary:monomial-sparse-interpolation}
or
Corollary~\ref{corollary:chebyshev-sparse-interpolation}.
The general idea is to perform the interpolation method repeatedly
on increasingly large intervals
and estimate the probability of having computed the
\qquot{true}
interpolating polynomial in terms of the
number of successive intervals with the same result
and a bound for the degree of~$f$.
For more details and further refinements we refer to~%
\cite{KL03}.

Early termination strategies can also be combined with
sparse interpolation methods for rational functions.
For details we refer to,~\eg,~%
Kaltofen-Yang~%
\cite{KY07}
and
Cuyt-Lee~%
\cite{CL11}.
In a related direction,
probabilistic methods tailored to sparse polynomial interpolation
over finite fields can be found,~\eg,~in~%
Arnold-Giesbrecht-Roche~%
\cite{AGR16}.

It would be interesting to look for generalizations of these results
in the framework of Prony structures.
However,
in full generality this is unlikely to be fruitful,
since one has to be able to make additional assumptions like degree bounds
for which the Prony structures are not well-adapted.

Another potential avenue for further research could be the investigation of
the computational complexity of Prony structures~\wrt~an underlying model of computation,
such as arithmetic circuits in polynomial identity testing.
See
Shpilka-Yehudayoff~%
\cite{SY10}
and
Saxena~%
\cite{Sax09,Sax14}
for recent surveys of this field.

We leave the search for suitable settings for the future.
\end{remark}

Now let
$A\in\real^{n\times n}$
be a fixed symmetric positive definite matrix.
A variant of Prony's method for $\complex$-linear combinations of the
\emph{Gau\ss{}ians}
\begin{equation*}
\text{${\gaussian_{A,t}}\colon\real^n\to\real$,
\quad
$x\mapsto\euler^{-\transpose{\rb{x-t}}A\rb{x-t}}$,
\quad
$t\in\real^n$,}
\end{equation*}
is proposed in
Peter-Plonka-Schaback~%
\cite{PPS15}.
In the following we identify the underlying Prony structure.
To this end,
let
\begin{equation*}
\text{$B\defeq\pset{{\gaussian_{A,t}}}{t\in\real^n}$
\quad
and
\quad
$V\defeq\mgen{B}_\complex$.}
\end{equation*}
For $t\in\real^n$
set
\begin{equation*}
\text{$b_{A,t}
\defeq
\euler^z
=
\rb{\euler^{z_1},\dots,\euler^{z_n}}
\in
\rb{\real\setminus\set{0}}^n$
with
$z
=
2\transpose{t}A
\in
\real^{1\times n}$}
\end{equation*}
and let
\begin{equation*}
\text{$u\colon B\to\real^n$,
\quad
${\gaussian_{A,t}}\mapsto
b_{A,t}$.}
\end{equation*}
Since~$A$~is positive definite,
$\gaussian_{A,t}$~obtains its unique maximum in~$t$.
This implies that~$u$~is well-defined.
Also since~$A$~is positive definite,
$b_{A,t}=b_{A,s}$ for $t,s\in\real^n$ implies that $t=s$,
and thus~$u$~is injective.
For the following theorem we set
$\widetilde{V}\defeq\Exp_{\integ,u\of{B}}^n\of{\complex}$,
$\widetilde{B}\defeq\pset{\exp_b}{b\in u\of{B}}$,
and
$\widetilde{u}\colon\widetilde{B}\to K^n$,
${\exp_b}\mapsto b$.
Recall that~$\widetilde{B}$~is a basis of~$\widetilde{V}$.

\begin{theorem}%
[Prony structure for Gau\ss{}ian sums]%
\label{theorem:gaussian-sums}
For $g\in V$
let
\begin{equation*}
\text{$f_g\colon\integ^n\to\complex$,
\quad
$\alpha\mapsto g\of{\alpha}\mul\euler^{\transpose{\alpha}A\alpha}$.}
\end{equation*}
Then the following holds:
\begin{tlist}
\item%
\label{theorem:gaussian-sums:phi(g)-exponential-sum}
For all $g\in V$
we have
$f_g\in\widetilde{V}$
and
$\varphi\colon V\to\widetilde{V}$,
$g\mapsto f_g$,
is a $\complex$-vector space isomorphism
with
$\varphi\of{\gaussian_{A,t}}=\lambda_{A,t}\mul\exp_{b_{A,t}}$
for some $\lambda_{A,t}\in\real\setminus\set{0}$.
In particular,
$B$~is a basis of~$V$.
\item%
\label{theorem:gaussian-sums:supp(phi(g))}
For all $g\in V$
we have
$\suppvar_{\widetilde{u}}\of{f_g}
=
\suppvar_u\of{g}$.
\end{tlist}
Hence,
any Prony structure on~$\widetilde{V}$
\trb{in particular
the Prony structures
from
Theorem~\ref{theorem:exponential-sums}},
induces a Prony structure on~$V$
by the transfer principle
\trb{Lemma~\ref{lemma:transfer-principle}}.
\end{theorem}

\begin{proof}
\ref{theorem:gaussian-sums:phi(g)-exponential-sum}
Note that
for all $t\in\real^n$ and $\alpha\in\integ^n$
and with
$\lambda_{A,t}
\defeq
\euler^{-\transpose{t}At}
\in\real\setminus\set{0}$
we have
\begin{equation*}
f_{\gaussian_{A,t}}\of{\alpha}
=
\gaussian_{A,t}\of{\alpha}\mul\euler^{\transpose{\alpha}A\alpha}
=
\euler^{-\transpose{\rb{\alpha-t}}A\rb{\alpha-t}}\mul\euler^{\transpose{\alpha}A\alpha}
=
\euler^{-\transpose{t}At}\mul\euler^{2\transpose{t}A\alpha}
=
\lambda_{A,t}\mul\exp_{b_{A,t}}\of{\alpha}\text{.}
\end{equation*}
By definition we have
$b_{A,t}
\in
u\of{B}$,
and
hence
$\varphi\of{\gaussian_{A,t}}
=
f_{\gaussian_{A,t}}
\in
\widetilde{V}$.
Since clearly
$f_{\lambda g+\mu h}=\lambda f_g+\mu f_h$
for all $\lambda,\mu\in\complex$ and $g,h\in V$,
we have that
$\varphi\of{V}\subseteq\widetilde{V}$
and~$\varphi$~is $\complex$-linear.
Since~$\widetilde{B}$~is a $\complex$-basis of~$\widetilde{V}$,
there is a unique $\complex$-linear map
$\psi\colon\widetilde{V}\to V$
with
$\psi\of{\exp_{b_{A,t}}}=1/\lambda_{A,t}\mul\gaussian_{A,t}$
for all
$t\in\real^n$.
Then~$\psi$~is the inverse of~$\varphi$
and this concludes the proof of~\ref{theorem:gaussian-sums:phi(g)-exponential-sum}.

\ref{theorem:gaussian-sums:supp(phi(g))}
Let
$g=\sum_{t\in F}\mu_t\gaussian_{A,t}$
with
finite
$F\subseteq\real^n$
and
$\mu_t\in\complex\setminus\set{0}$.
Using
part~\ref{theorem:gaussian-sums:phi(g)-exponential-sum}
we obtain
\begin{equation*}
\suppvar_{\widetilde{u}}\of{f_g}
=
\suppvar_{\widetilde{u}}\sumof{\sum_{t\in F}\mu_t\lambda_{A,t}\exp_{b_{A,t}}}
=
\pset{b_{A,t}}{t\in F}
=
\suppvar_u\of{g}\text{,}
\end{equation*}
\idest,~the assertion.
\end{proof}

Note that
an alternative approach
to the reconstruction problem in
Theorem~\ref{theorem:gaussian-sums}
which is based on Fourier transforms
is proposed in
Peter-Potts-Tasche~%
\cite{PPT11}.

\begin{remark}
There is a close relationship between Prony's method
and Sylvester's method for computing Waring decompositions of homogeneous polynomials.
Although Sylvester's method does not fit directly into our framework
of Prony structures
\trb{since it is not a method to reconstruct the support of a function},
one may still view it as an application of the Prony structure from
Example~\ref{example:classic-prony}:
Given a homogeneous polynomial
\begin{equation*}
p
=
\sum_{i=0}^dp_i\X^i\Y^{d-i}
\in
\complex\ringad{\X,\Y}\text{,}
\end{equation*}
of Waring rank at most~$r$,
then the matrix
\begin{equation*}
\Cat\of{p}
\defeq
\rb{c_{i+j}}_{\substack{i=0,\dots,r\\j=0,\dots,d-r}}
\in
\complex^{\rb{r+1}\times\rb{d-r+1}}
\end{equation*}
with
$c_i\defeq p_i/\binom{d}{i}$
induces a Prony structure
for an exponential sum
\trb{in the sense that $\ker\Cat\of{p}$ identifies the support}.
Then this exponential sum $f_p\in\Exp^1\of{\complex}$
and its reconstruction as
$f_p=\sum_{k=1}^r\mu_k\exp_{b_k}$
can be used to compute a Waring decomposition of~$p$.
Sylvester's method has recently been generalized to the multivariate case,
\confer~%
\cite{BCMT10}.
\end{remark}

\section{Relative Prony structures}%
\label{section:relative-prony-structures}

A Prony structure on a vector space~$V$~can be seen as a tool to obtain polynomials
that identify the $u$-support $\suppvar_u\of{f}\subseteq K^n$ of a given $f\in V$.
Suppose that we are given~\apriori~a~set of polynomials $I\subseteq S=K\ringad{\X_1,\dots,\X_n}$
with $\suppvar_u\of{f}\subseteq\ZL\of{I}$.
For example,
one could have $K=\real$ and $\suppvar_u\of{f}\subseteq\sphere^{n-1}=\ZL\of{1-\sum_{j=1}^n\X_j^2}$.
Prony structures as previously discussed do not take this additional information into account.
In this section we extend Prony structures in order to take advantage of this situation.

We begin by giving appropriate variants of earlier definitions for this context.

\begin{definition}
For
$Y\subseteq K^n$
let
\begin{equation*}
K\ringad{Y}
\defeq
K\ringad{\X}/{\I\of{Y}}
\end{equation*}
be the usual
\emph{coordinate algebra of~$Y$}.
For
$D\subseteq\nat^n$
let,
as before,
$\X^D
=
\pset{\X^\alpha}{\alpha\in D}$
and
\begin{equation*}
\lbar{\X^D}\defeq\pset{m+\I\of{Y}}{m\in\X^D}\subseteq K\ringad{Y}\text{.}
\end{equation*}
We denote by
\begin{equation*}
K\ringad{Y}_D
\defeq
\mgen{\lbar{\X^D}}_K
\end{equation*}
the $K$-subvector space of~$K\ringad{Y}$
generated by~$\lbar{\X^D}$.
We call~$K\ringad{Y}_D$
the
\emph{coordinate space of~$Y$~\wrt~$S_D$}.
\end{definition}

\begin{remark}
Let
$Y\subseteq K^n$
and
$D\subseteq\nat^n$.
Then we have
\begin{equation*}
K\ringad{\X}_D/{\I_D\of{Y}}\isom K\ringad{Y}_D\text{.}
\end{equation*}
Indeed,
the $K$-linear map $K\ringad{\X}_D\to K\ringad{Y}_D$
with~$\X^\alpha\mapsto\lbar{\X^\alpha}=\X^\alpha+\I\of{Y}$
for $\alpha\in D$
is an epimorphism with kernel~$\I_D\of{Y}$.
In the following we identify these two $K$-vector spaces.
\end{remark}

\begin{definition}
Let $D\subseteq\nat^n$.
For
$X\subseteq Y\subseteq K^n$
we call
\begin{equation*}
\text{${\ev{D/Y}{X}}\colon K\ringad{Y}_D\to K^X$,
\quad
$p+\I_D\of{Y}\mapsto\ev{D}{X}\of{p}=\rb{p\of{x}}_{x\in X}$,}
\end{equation*}
the
\emph{relative evaluation map at~$X$~\wrt~$S_D$~modulo~$Y$}
and
\begin{equation*}
\I_{D/Y}\of{X}
\defeq
\ker\of{\ev{D/Y}{X}}
\end{equation*}
the
\emph{relative vanishing space of~$X$~\wrt~$S_D$~modulo~$Y$}.
\end{definition}

\begin{remark}%
\label{remark:relative-evaluation-map}
Let
$X\subseteq Y\subseteq K^n$
and
$D\subseteq\nat^n$,
$X$~and~$D$~finite.
Since~$\lbar{\X^D}$~generates~$K\ringad{Y}_D$
there is a $C\subseteq D$
such that~$\lbar{\X^C}$~is a $K$-basis of~$K\ringad{Y}_D$.
Without loss of generality,
choose~$C$~such that $\card{\lbar{\X^C}}=\card{C}$.

Observe that then the transformation matrix
of~$\ev{D/Y}{X}$~\wrt~$\lbar{\X^C}$~and the canonical basis
of~$K^X$~is
the Vandermonde matrix
$\Vandermonde_C^X=\rb{x^\alpha}_{x\in X,\alpha\in C}$.
Hence the transformation matrices of the relative evaluation map~$\ev{D/Y}{X}$~and
the \qquot{ordinary} evaluation map~$\ev{C}{X}$~are identical.
\end{remark}

\begin{definition}
For
$J\subseteq K\ringad{Y}$
we call
\begin{equation*}
\ZL_Y\of{J}
\defeq
\pset{y\in Y}{\text{for all $q\in S$ with $q+\I\of{Y}\in J$, $q\of{y}=0$}}
\end{equation*}
the
\emph{relative zero locus of~$J$~\wrt~$Y$}.
\end{definition}

After these general preparations,
we define relative Prony structures,
which are the topic of this section.
Recall that an
\emph{algebraic set}
$Y\subseteq K^n$
is the zero~locus of a set of polynomials,
\idest,~$Y=\ZL\of{I}$
for some set of polynomials $I\subseteq S$.
By Hilbert's basis theorem,~$I$~can always be chosen to be finite.

\begin{definition}
% \label{definition:relative-prony-structure}
Given the setup of
Definition~\ref{definition:u-support},
let
$Y\subseteq K^n$
be an algebraic set,
and suppose that
\begin{equation*}
u\of{B}
\subseteq
Y\text{.}
\end{equation*}

Let
$\FI=\rb{\FI_d}_{d\in\nat}$
be a sequence of finite sets
and
$\FH=\rb{\FH_d}_{d\in\nat}$
be a sequence of finite subsets of~$\nat^n$
such that
$\card{\lbar{\X^{\FH_d}}}=\card{\FH_d}$
and the vectors in the set~$\lbar{\X^{\FH_d}}$~are linearly independent in~$K\ringad{Y}$.

Let
$f\in V$
and
\begin{equation*}
P\of{f}=\rb{P_d\of{f}}_{d\in\nat}\in\prod_{d\in\nat}K^{\FI_d\times\FH_d}\text{.}
\end{equation*}

We call~$P\of{f}$~a
\emph{\trb{relative} Prony structure~\wrt~$Y$~for~$f$}
if for all large~$d$~one has
\begin{equation}%
\label{equation:definition:relative-prony-structure}
\text{$\ZL_Y\of{\ker P_d\of{f}}=\suppvar_u\of{f}$
\quad
and
\quad
$\I_{\FH_d/Y}\of{\suppvar_u\of{f}}\subseteq\ker\of{P_d\of{f}}$.}
\end{equation}
Here we identify
$p\in\ker P_d\of{f}\subseteq K^{\FH_d}$
with
$\sum_{\alpha\in\FH_d}p_\alpha\lbar{\X^\alpha}
\in
K\ringad{Y}_{\FH_d}
\le
K\ringad{Y}$.

The least $c\in\nat$
such that
the conditions in~\eqref{equation:definition:relative-prony-structure}
hold for all $d\ge c$
is called
\emph{\trb{relative} Prony index~\wrt~$Y$~of~$f$}
or simply
\emph{$P$-index~\wrt~$Y$~of~$f$},
denoted by
$\Pronyindex_{P,Y}\of{f}$.

If for every $f\in V$ a relative Prony structure~$P\of{f}$~\wrt~$Y$~for~$f$~is given,
then we call~$P$~a
\emph{\trb{relative} Prony structure~\wrt~$Y$~on~$V$}.
\end{definition}

\begin{remark}
Over an infinite field~$K$,
Prony structures as considered before
are precisely the relative Prony structures~\wrt~$Y=K^n$.
This follows immediately from
$K\ringad{Y}=K\ringad{\X}$.
\end{remark}

We obtain a characterization of relative Prony structures
analogous to one for ordinary Prony structures in
Theorem~\ref{theorem:prony-structure-characterization}.

\begin{theorem}%
[Relative version of
Theorem~\ref{theorem:prony-structure-characterization}]%
\label{theorem:relative-prony-structure-characterization}
Given the setup of
Definition~\ref{definition:u-support},
let
$f\in V$,
$B$~an~$F$-basis of~$V$,
$u\colon B\to K^n$ injective,
$\FI$~a sequence of finite sets,
and
$\FJ$~a sequence of finite subsets of~$\nat^n$
with
$\FJ_d\subseteq\FJ_{d+1}$
for all large~$d$
and
$\bigcup_{d\in\nat}\FJ_d=\nat^n$.
Let
$Y\subseteq K^n$
be an algebraic set
with
\begin{equation*}
\suppvar_u\of{f}\subseteq Y
\end{equation*}
and
$\FH_d\subseteq\FJ_d$
such that~$\lbar{\X^{\FH_d}}$~is
a $K$-basis of~$K\ringad{Y}_{\FJ_d}$
with $\card{\lbar{\X^{\FH_d}}}=\card{\FH_d}$.
Let
\begin{equation*}
Q\in\prod_{d\in\nat}K^{\FI_d\times\FH_d}\text{.}
\end{equation*}
Then the following are equivalent:
\begin{ifflist}
\item%
\label{theorem:relative-prony-structure-characterization:relative-prony-structure}
$Q$~is a Prony structure~\wrt~$Y$~for~$f$;
\item%
\label{theorem:relative-prony-structure-characterization:monomorphism}
For all large~$d$
there is
an injective $K$-linear map $\eta_d\colon K^{\suppvar_u\of{f}}\emb K^{\FI_d}$
such that the diagram
\begin{equation*}
\begin{tikzcd}
|[alias=KHd ]|  K^{\FH_d}            &&&  |[alias=KId  ]|  K^{\FI_d}             \\
|[alias=KYHd]|  K\ringad{Y}_{\FH_d}  &&&  |[alias=Ksupp]|  K^{\suppvar_u\of{f}}
\ar[  "Q_d"                             ,  from=KHd    ,  to=KId    ,          ,        ]
\ar[  "\isom"                           ,  from=KHd    ,  to=KYHd   ,          ,        ]
\ar[  "\ev{\FH_d/Y}{\suppvar_u\of{f}}"  ,  from=KYHd   ,  to=Ksupp  ,          ,        ]
\ar[  "\eta_d"'                         ,  from=Ksupp  ,  to=KId    ,  dashed  ,  hook  ]
\end{tikzcd}
\end{equation*}
is commutative;
\item
% \label{theorem:relative-prony-structure-characterization:kernel}
For all large~$d$
we have
$\ker\of{Q_d}=\I_{\FH_d/Y}\of{\suppvar_u\of{f}}$.
\end{ifflist}
\end{theorem}

\begin{proof}
Using
Remark~\ref{remark:relative-evaluation-map}
for the surjectivity
$\ev{\FH_d/Y}{\suppvar_u\of{f}}$
for all large~$d$,
the proof is analogous to the one of
Theorem~\ref{theorem:prony-structure-characterization}.
\end{proof}

The following theorem gives a method
to obtain a relative Prony structure from an \qquot{ordinary} one.
The relative Prony structure then uses smaller matrices.

\begin{theorem}%
\label{theorem:relative-prony-structures-construction}
Let~$P$~be a Prony structure on~$V$~as defined in
Definition~\ref{definition:prony-structure}
and let
$Y\subseteq K^n$
be an algebraic set
with
$u\of{B}\subseteq Y$.
Let
$\FH_d\subseteq\nat^n$
be such that~$\lbar{\X^{\FH_d}}$~is
a $K$-basis of~$K\ringad{Y}_{\FJ_d}\le K\ringad{Y}$
and
$\card{\FH_d}=\card{\lbar{\X^{\FH_d}}}$.
Let
\begin{equation*}
\text{$P_\FH\colon V\to\prod_{d\in\nat}K^{\FI_d\times\FH_d}$,
\quad
$f\mapsto\rb{\rb{P_\FH}_d\of{f}}_{d\in\nat}
\defeq
\rb{\restr{P_d\of{f}}{\rb{\FI_d\times\FH_d}}}_{d\in\nat}$.}
\end{equation*}
Here~$\restr{P_d\of{f}}{\rb{\FI_d\times\FH_d}}$~is obtained from~$P_d\of{f}$~by
deleting the columns that are not in~$\FH_d$.
Then~$P_\FH$~induces a Prony structure~\wrt~$Y$~on~$V$.
\end{theorem}

\begin{proof}
Let $f\in V$.
By
Theorem~\ref{theorem:prony-structure-characterization},
for all large~$d$
there are injective $K$-linear maps
$\eta_d\colon K^{\suppvar_u\of{f}}\emb K^{\FI_d}$
such that
the linear map
$S_{\FJ_d}\to K^{\FI_d}$
induced by~$P_d\of{f}$~equals
$\eta_d\comp{\ev{\FJ_d}{\suppvar_u\of{f}}}$.

It is easy to see that then
the linear map
$S_{\FH_d}\to K^{\FI_d}$
induced by
$\rb{P_\FH}_d\of{f}$
equals
$\eta_d\comp{\ev{\FH_d}{\suppvar_u\of{f}}}$
\trb{for all large~$d$}.
Recall that the matrix
of
$\ev{\FH_d}{\suppvar_u\of{f}}$
equals
$\Vandermonde_{\FH_d}^{\suppvar_u\of{f}}$.
Thus,
we have
$\rb{P_\FH}_d\of{f}
=
E_d\mul\Vandermonde_{\FH_d}^{\suppvar_u\of{f}}$
where~$E_d$~denotes the matrix of~$\eta_d$.
Hence,
by
Remark~\ref{remark:relative-evaluation-map}
we have that the linear map
$K\ringad{Y}_{\FH_d}\to K^{\FI_d}$
induced by~$\rb{P_\FH}_d\of{f}$
equals
$\eta_d\comp{\ev{\FH_d/Y}{\suppvar_u\of{f}}}$.
By the direction~\qquot{\claimmimplnocolon{\ref{theorem:relative-prony-structure-characterization:monomorphism}}{\ref{theorem:relative-prony-structure-characterization:relative-prony-structure}}}
of
Theorem~\ref{theorem:relative-prony-structure-characterization}
we are done.
\end{proof}

\begin{corollary}%
[Relative version of
Theorem~\ref{theorem:exponential-sums}\,\ref{theorem:exponential-sums:hankel-NN^n}]%
\label{corollary:exponential-sums-relative}
Let
$V\defeq\Exp^n_Y\of{F}$
with an algebraic set $Y\subseteq K^n$.
For appropriately chosen sequences~$\FI$~and~$\FJ$,
\begin{equation*}
\Hankel_{\FI,\FJ,d}\of{f}=\rb{f\of{\alpha+\beta}}_{\substack{\alpha\in\FI_d\\\beta\in\FJ_d}}
\in
K^{\FI_d\times\FJ_d}
\end{equation*}
induces a Prony structure on~$V$
according to
Theorem~\ref{theorem:exponential-sums}\,\ref{theorem:exponential-sums:hankel-NN^n}.
Let $\FH_d\subseteq\FJ_d$
be such that~$\lbar{\X^{\FH_d}}$~is a $K$-basis of~$K\ringad{Y}_{\FJ_d}$.
Then
\begin{equation*}
\Hankel_{\FI,\FH,d}\of{f}
=
\rb{f\of{\alpha+\beta}}_{\substack{\alpha\in\FI_d\\\beta\in\FH_d}}
\in
K^{\FI_d\times\FH_d}
\end{equation*}
induces a Prony structure~\wrt~$Y$~on~$V$.
\end{corollary}

\begin{proof}
This is immediate by applying
Theorem~\ref{theorem:relative-prony-structures-construction}
to
Theorem~\ref{theorem:exponential-sums}\,\ref{theorem:exponential-sums:hankel-NN^n}.
\end{proof}

\begin{remark}%
\label{remark:exponential-sums-relative}
~
\begin{tlist}
\item
An analogous result to
Corollary~\ref{corollary:exponential-sums-relative}
holds for the Toeplitz Prony structure
on~$V=\Exp_Y^n\of{F}$
for an \trb{algebraic} set $Y\subseteq\rb{K\setminus\set{0}}^n$.
\item%
\label{remark:exponential-sums-relative:row-deletion}
For
$V=\Exp^n_Y\of{F}$
as in
Corollary~\ref{corollary:exponential-sums-relative}
a more efficient result is possible
as follows.

As a matrix,
$\Hankel_{\FI,\FH,d}\of{f}$~is obtained by \qquot{deleting columns} from~$\Hankel_{\FI,\FJ,d}\of{f}$.
By the proof of
Theorem~\ref{theorem:exponential-sums}\,\ref{theorem:exponential-sums:hankel-NN^n},
the linear map
$S_{\FH_d}\to K^{\FI_d}$
induced by
$\Hankel_{\FI,\FH,d}\of{f}$
equals
$\eta_{\FI,d}\comp{\ev{\FH_d}{\suppvar_u\of{f}}}$
with
$\eta_{\FI,d}\defeq\transpose{\rb{\ev{\FI_d}{\suppvar_u\of{f}}}}\comp C_f$.
Thus,
we may also pass to
$\eta_{\FH,d}\defeq\transpose{\rb{\ev{\FH_d}{\suppvar_u\of{f}}}}\comp C_f$,
since also~$\eta_{\FH,d}$~is injective for all large~$d$.
Hence also
$\Hankel_{\FH,\FH,d}\of{f}=\rb{f\of{\alpha+\beta}}_{\alpha,\beta\in\FH_d}$
induces a Prony structure~\wrt~$Y$~on $V=\Exp^n_Y\of{F}$.
\end{tlist}
\end{remark}

While
Theorem~\ref{theorem:relative-prony-structures-construction}
yields a general recipe to construct relative Prony structures
from \qquot{ordinary} ones,
in concrete situations it can be possible to achieve better results.
We end the section with one such example,
recasting the main result of~%
\cite{KMv19}
in the context of relative Prony structures.
Let
$K=\real$,
$S=\real\ringad{\X_1,\dots,\X_n}$,
and
\begin{equation*}
Y
\defeq
\sphere^{n-1}
=
\ZL\sumof{1-\sum_{j=1}^n\X_j^2}
=
\pset{x\in\real^n}{\norm{x}_2=1}
\subseteq
\real^n\text{.}
\end{equation*}
Consider the $\real$-vector space
\begin{equation*}
\sphereharmon_{\le d}
\defeq
\real\ringad{\sphere^{n-1}}_{\le d}
=
S_{\le d}/{\I_{\le d}\of{\sphere^{n-1}}}
\isom
\pset{\restr{p}{\sphere^{n-1}}}{p\in S_{\le d}}\text{.}
\end{equation*}

Let
${\laplace}\colon S\to S$,
$p\mapsto\sum_{j=1}^n\partial_j^2\of{p}$,
denote the
\emph{Laplace operator}.
The elements of
$\ker\of{\laplace}$ are called
\emph{harmonic}.

Let~$\harmonhomog_k$~be the $\real$-vector space
generated by the restrictions~$\restr{p}{\sphere^{n-1}}$~of
harmonic homogeneous polynomials $p\in S_k$
of degree~$k$~to the sphere,
usually called
the space of
\emph{spherical harmonics}.
Using
Gallier-Quaintance~%
\cite[Theorem~\textup{7.13},
discussion after
Definition~\textup{7.15}]{GQ20}
it is easy to see that one has the decomposition
\trb{as vector spaces}
\begin{equation*}
\sphereharmon_{\le d}
\isom
\bigdirsum_{k=0}^d\harmonhomog_k\text{.}
\end{equation*}
For $k=0,\dots,d$,
let
$H_k=\rb{y_k^1,\dots,y_k^{d_k}}$
be an $\real$-basis of~$\harmonhomog_k$.
Hence
$H_{\le d}\defeq\bigcup_{k=0}^dH_k$
is a basis of~$\sphereharmon_{\le d}$.
For $x\in\sphere^{n-1}$
let
\begin{equation*}
\text{$h_x\colon\pset{\rb{k,\ell}}{\text{$k\in\nat$, $\ell=1,\dots,d_k$}}\to\real$,
\quad
$\rb{k,\ell}\mapsto y_k^\ell\of{x}$.}
\end{equation*}
For finite $X\subseteq\sphere^{n-1}$
let~$W_{\le d}^X$~be the matrix of
${\ev{\FT_d/\sphere^{n-1}}{X}}$~\wrt~$H_{\le d}$~and the basis of~$\real^X$.

\begin{corollary}%
[Relative Prony structure for spherical harmonic sums]%
\label{corollary:spherical-harmonics}
Let
$B\defeq\pset{h_x}{x\in\sphere^{n-1}}$,
$V\defeq\mgen{B}_\real$,
and
$u\colon B\to\real^n$,
$h_x\mapsto x$.
For
$f\in V$,
$f=\sum_{x\in\suppvar_u\of{f}}f_xh_x$,
$f_x\in\real\setminus\set{0}$,
let
$C_f=\rb{f_x\unit_x}_{x\in X}$
and
\begin{equation*}
\widetilde{\Hankel}_d\of{f}
=
\transpose{\rb{W^{\suppvar_u\of{f}}_{\le d}}}\mul C_f\mul W^{\suppvar_u\of{f}}_{\le d}\text{.}
\end{equation*}

Then the function
\begin{equation*}
\text{$\widetilde{\Hankel}\colon V\to\prod_{d\in\nat}\real^{H_{\le d}\times H_{\le d}}$,
\quad
$f\mapsto\rb{\widetilde{\Hankel}_d\of{f}}_{d\in\nat}$,}
\end{equation*}
induces a relative Prony structure~\wrt~$\sphere^{n-1}$~on~$V$.
\end{corollary}

\begin{proof}
This follows from
Kunis-M\"oller-von~der~Ohe~%
\cite[Section~\textup{3.3},
Theorem~\textup{3.14}]{KMv19}.
\end{proof}

\begin{remark}
Observe that by~%
\cite[Theorem~\textup{3.14}]{KMv19},
the matrix~$\widetilde{\Hankel}_d\of{f}$~can be computed solely from
$\bigTheta\of{d^{n-1}}$
evaluations of~$f$.
One may also use
Corollary~\ref{corollary:exponential-sums-relative}
or even
Remark~\ref{remark:exponential-sums-relative}\,\ref{remark:exponential-sums-relative:row-deletion}
to get a Prony structure~\wrt~$\sphere^{n-1}$
on~$\sphereharmon_{\le d}$.
The matrices so obtained have the same number of columns
or the same size
as the
ones in
Corollary~\ref{corollary:spherical-harmonics},
respectively.
But then the number~$\card{\FH_d+\FH_d}$~of used evaluations
is not in general in~$\bigTheta\of{d^{n-1}}$.
\end{remark}

\section{Maps between Prony structures}%
\label{section:prony-maps}

In
Section~\ref{section:applications}
we witnessed instances of Prony structures transferring
from one vector space to another,
such as from spaces of exponential sums to spaces of polynomials
or Gau\ss{}ian sums
with their respective bases.
We take these observations as motivation
to consider structure preserving maps between Prony structures.
For notational simplicity,
whenever we say that~$P$~is a Prony structure,
we mean that~$P$~is a Prony structure on
an $F$-vector space~$V$~with basis~$B$~\wrt~an injection $u\colon B\to K^n$.
Similarly,
when~$P^\prime$~is a Prony structure,
then this means that~$P^\prime$~is a Prony structure on
an $F^\prime$-vector space~$V^\prime$~with basis~$B^\prime$~\wrt~an injection $u^\prime\colon B^\prime\to\rb{K^\prime}^{n^\prime}$.

The following is natural definition of structures preserving maps between Prony structures.

\begin{definition}
Let~$P$~and~$P^\prime$~be Prony structures on~$V$~and~$V^\prime$,~respectively.
Let
\begin{itemize}
\item
$\iota\colon F\to F^\prime$
be
a field homomorphism
\trb{turning~$V^\prime$ into an $F$-vector space},
\item
$\varphi\colon V\to V^\prime$
be
an $F$-vector space homomorphism,
and
\item
$\mu\colon P\of{V}\to P^\prime\of{V^\prime}$
be
a function,
where
$P\of{V}=\pset{\rb{P_d\of{f}}_{d\in\nat}}{f\in V}$.
\end{itemize}
Then
$\psi\defeq\rb{\iota,\varphi,\mu}$
is called
\emph{map of Prony structures from~$P$ to~$P^\prime$},
abbreviated as
\emph{Prony map}
in the following,
written $\psi\colon P\to P^\prime$,
if
the inclusion
\begin{equation*}
\varphi\of{B}
\subseteq
B^\prime
\end{equation*}
holds
and
the following diagram is commutative:
\begin{equation*}
\begin{tikzcd}
|[alias=V ]|  V         &  |[alias=PV  ]|  P\of{V}                \\
|[alias=V']|  V^\prime  &  |[alias=P'V']|  P^\prime\of{V^\prime}
\ar[  "P"         ,  from=V   ,  to=PV    ]
\ar[  "P^\prime"  ,  from=V'  ,  to=P'V'  ]
\ar[  "\mu"       ,  from=PV  ,  to=P'V'  ]
\ar[  "\varphi"   ,  from=V   ,  to=V'    ]
\end{tikzcd}
\end{equation*}
\end{definition}

\begin{remark}
Our notation should not be confused with a similar definition in
Batenkov-Yomdin~%
\cite{BY14}
where certain moment maps are considered.

One might expect a map
between~$K^n$ and~$\rb{K^\prime}^{n^\prime}$
in the definition of Prony map
\trb{that is compatible with the other data}.
However,
if~$P$~and~$P^\prime$~are Prony structures
and
$\psi=\rb{\iota,\varphi,\mu}\colon P\to P^\prime$
is a Prony map
then,
since~$u$~is injective,
there is
always
a function
\begin{equation*}
\text{$\varrho_\psi\colon u\of{B}\to u^\prime\of{B^\prime}$,
\quad
$\ell\mapsto\rb{u^\prime\comp\varphi}\of{u^{-1}\of{\ell}}$,}
\end{equation*}
that maps
elements of
$u\of{B}\subseteq K^n$
to elements of
$u^\prime\of{B^\prime}\subseteq\rb{K^\prime}^{n^\prime}$.
In other words,
the following diagram is commutative:
\begin{equation*}
\begin{tikzcd}
|[alias=B ]|  B        &  |[alias=B'  ]|  B^\prime               \\
|[alias=uB]|  u\of{B}  &  |[alias=u'B']|  u^\prime\of{B^\prime}
\ar[  "\varphi"       ,  from=B   ,  to=B'    ,        ,             ]
\ar[  "\varrho_\psi"  ,  from=uB  ,  to=u'B'  ,        ,             ]
\ar[  "u"             ,  from=B   ,  to=uB    ,  hook  ,  two heads  ]
\ar[  "u^\prime"      ,  from=B'  ,  to=u'B'  ,  hook  ,  two heads  ]
\end{tikzcd}
\end{equation*}
Clearly,
$\varrho_\psi$~is injective if and only if~$\varphi$~is injective.
\end{remark}

\begin{remark}
Let
$\catP=\rb{\calO,{\Hom},{\id},{\comp}}$
be defined as follows.
\begin{itemize}
\item
$\calO\defeq\pset{P}{\text{$P$ Prony structure}}$
is the class of all Prony structures.
\item
For $P,P^\prime\in\calO$,
$\Hom\of{P,P^\prime}\defeq\pset{\psi}{\text{$\psi\colon P\to P^\prime$ Prony map}}$
is the set of all Prony maps from~$P$~to~$P^\prime$.
\item
For
$P\in\calO$,
let
\begin{equation*}
{\id_P}
\defeq
\rb{{\id_F},{\id_V},{\id_{P\of{V}}}}\text{.}
\end{equation*}
\item
For
$P,P^\prime,P^{\prime\prime}\in\calO$,
$\psi
=
\rb{\iota,\varphi,\mu}\in\Hom\of{P,P^\prime}$,
and
$\psi^\prime
=
\rb{\iota^\prime,\varphi^\prime,\mu^\prime}\in\Hom\of{P^\prime,P^{\prime\prime}}$,
let
\begin{equation*}
\psi^\prime\comp\psi
\defeq
\rb{\rb{\iota^\prime\comp\iota},\rb{\varphi^\prime\comp\varphi},\rb{\mu^\prime\comp\mu}}\text{.}
\end{equation*}
\end{itemize}
It is straightforward to show
that~$\catP$~is a category
\trb{\confer,~\eg,~%
\cite{Mac98,AHS05}}.
We call~$\catP$~the
\emph{category of Prony structures}.
It would be interesting to get insights from this point of view.
\end{remark}

\begin{example}%
[Sparse polynomial interpolation]
Let the notation and assumptions be as in
Corollary~\ref{corollary:monomial-sparse-interpolation},
and moreover let
$\iota
\defeq
{\id_F}$
be the identity map on~$F$.
Note that
\begin{equation*}
Q_P\of{V}
=
\pset{Q_P\of{p}}{p\in V}
=
\pset{P\of{f_p}}{p\in V}
\subseteq
P\of{\widetilde{V}}\text{.}
\end{equation*}
So we choose
$\mu\colon Q_P\of{V}\emb P\of{\widetilde{V}}$
to be the inclusion map.
Then
$\psi\defeq\rb{\iota,\varphi,\mu}$
is a Prony map
from~$Q_P$~to~$P$.
Indeed,
easy computations show that
$\varphi\colon V\to\widetilde{V}$
is a vector space homomorphism
and that
$\mu\comp Q_P=P\comp\varphi$.
\end{example}

\begin{example}%
[Projection methods]%
\label{example:projection-prony-map}
For
$n\in\nat$
let
$V_n\defeq\Exp^n_{K^n}\of{F}$.
Let~$\Hankel_n$
be the Prony structure
derived from
Theorem~\ref{theorem:exponential-sums}\,\ref{theorem:exponential-sums:hankel-NN^n}.

For
a fixed
$\alpha\in\nat^n$
let
\begin{equation*}
\text{$\varphi_\alpha\colon V_n\to V_1$,
\quad
$f\mapsto f_\alpha$,}
\end{equation*}
where
\begin{equation*}
\text{$f_\alpha\colon\nat\to K$,
\quad
$k\mapsto f\of{k\mul\alpha}$.}
\end{equation*}
It is easy to see that $f_\alpha\in V_1$ and hence~$\varphi$~is well-defined.
Furthermore,
let
\begin{equation*}
\text{$\mu_\alpha\colon{\Hankel_n\of{V_n}}\to{\Hankel_1\of{V_1}}$,
\quad
$\Hankel_n\of{f}\mapsto\Hankel_1\of{f_\alpha}$.}
\end{equation*}
Then
$\psi_\alpha\defeq\rb{{\id_F},\varphi_\alpha,\mu_\alpha}$
is a Prony map from~$\Hankel_n$~to~$\Hankel_1$.

Also note that
$\Hankel_{1,d}\of{f_\alpha}
=
\rb{\Hankel_{n,d}\of{f}_{\beta,\gamma}}_{\beta,\gamma\in\FJ_{1,d}\mul\alpha}$.
\end{example}

\begin{proof}
It is easy to verify that~$\varphi_\alpha$~is $F$-linear.
Furthermore,
for every $b\in K^n$ we have
\begin{equation*}
\varphi_\alpha\of{\exp_b}
=
\exp_{b^\alpha}\text{,}
\end{equation*}
hence $\varphi_\alpha\of{B_n}\subseteq B_1$.
The identity
${\Hankel_1}\comp\varphi_\alpha=\mu_\alpha\comp{\Hankel_n}$
holds by the definitions.

Finally,
let
$f\in V_n$
and
$d\in\nat$.
We have
\begin{align*}
\Hankel_{1,d}\of{f_\alpha}
&=
\rb{f_\alpha\of{k+\ell}}_{k,\ell\in\FJ_{1,d}}
=
\rb{f\of{\rb{k+\ell}\alpha}}_{k,\ell\in\FJ_{1,d}}
=
\rb{f\of{k\alpha+\ell\alpha}}_{k,\ell\in\FJ_{1,d}}\\
&=
\rb{f\of{\beta+\gamma}}_{\beta,\gamma\in\FJ_{1,d}\mul\alpha}
=
\rb{\Hankel_{n,d}\of{f}_{\beta,\gamma}}_{\beta,\gamma\in\FJ_{1,d}\mul\alpha}\text{,}
\end{align*}
which concludes the proof.
\end{proof}

\begin{remark}
For the Prony structures~$\Toeplitz_n$~and~$\Hankel_n$~from
Theorem~\ref{theorem:exponential-sums}\,\ref{theorem:exponential-sums:toeplitz-ZZ^n},~\ref{theorem:exponential-sums:hankel-ZZ^n}
Prony maps ${\Toeplitz_n}\to{\Toeplitz_1}$
and ${\Hankel_n}\to{\Hankel_1}$
can be constructed analogously to
Example~\ref{example:projection-prony-map}.
\end{remark}

\begin{example}
There is a Prony map
$\psi=\rb{\iota,\varphi,\mu}\colon{\Toeplitz}\to{\Hankel}$
given by
$\iota={\id_F}$,
$\varphi={\id_{\Exp^n_Y\of{F}}}$,
and
$\mu\of{\Toeplitz\of{f}}=\Hankel\of{f}$.
Note that~$\mu$~is well-defined
since all the coefficients of
$\Hankel_d\of{f}=\rb{f\of{\alpha+\beta}}_{\alpha,\beta\in\FJ_d}$
appear in the matrix
$\Toeplitz_e\of{f}=\rb{f\of{\alpha-\beta}}_{\beta,\alpha\in\FJ_e}$
for some $e\in\nat$.
\end{example}

\begin{example}%
[Gau\ss{}ian sums]
Let the notation and assumptions be as in
Theorem~\ref{theorem:gaussian-sums}
and let
$\widetilde{C}\defeq\varphi\of{B}=\pset{\lambda_{A,t}\mul\exp_{b_{A,t}}}{t\in\real^n}$.
Clearly,
$\widetilde{C}$~is a basis of~$\widetilde{V}$.
Let
$\widetilde{v}\colon\widetilde{C}\to\real^n$,
$\varphi\of{b}\mapsto\widetilde{u}\of{b}$
and let~$P$~be any Prony structure on~$\widetilde{V}$~\wrt~$\widetilde{v}$.
Let $\iota\colon\complex\to\complex$ be the identity map.
It is again easy to see that
$Q_P\of{V}\subseteq P\of{\widetilde{V}}$.
Thus,
let
$\mu\colon Q_P\of{V}\emb P\of{\widetilde{V}}$
be the inclusion map.
Then
$\psi\defeq\rb{\iota,\varphi,\mu}\colon Q_P\to P$
is a Prony map.
Indeed,
we have already seen
in
Theorem~\ref{theorem:gaussian-sums}
that
$\varphi\colon V\to\widetilde{V}$
is a $\complex$-vector space isomorphism.
By the definitions,
we have
$\varphi\of{B}\subseteq\widetilde{C}$
and
the diagram
\begin{equation*}
\begin{tikzcd}
|[alias=V ]|  V              &  |[alias=PV  ]|  Q_P\of{V}            \\
|[alias=V']|  \widetilde{V}  &  |[alias=P'V']|  P\of{\widetilde{V}}
\ar[  "Q_P"      ,  from=V   ,  to=PV    ]
\ar[  "P"        ,  from=V'  ,  to=P'V'  ]
\ar[  "\mu"      ,  from=PV  ,  to=P'V'  ]
\ar[  "\varphi"  ,  from=V   ,  to=V'    ]
\end{tikzcd}
\end{equation*}
is commutative.
\end{example}

\end{document}